\documentclass[review,hidelinks,onefignum,onetabnum]{siamart220329}
\newtheorem{example}{Example}[section]

\usepackage{bm}
\usepackage{lipsum}
\usepackage{amsfonts}
\usepackage{graphicx}
\usepackage{epstopdf}
\usepackage{algorithmic}
\usepackage{comment}
\ifpdf
  \DeclareGraphicsExtensions{.eps,.pdf,.png,.jpg}
\else
  \DeclareGraphicsExtensions{.eps}
\fi

\newsiamremark{remark}{Remark}
\newsiamremark{hypothesis}{Hypothesis}
\crefname{hypothesis}{Hypothesis}{Hypotheses}
\newsiamthm{claim}{Claim}

\headers{QQMR: A Structure Preserving Quaternion Quasi-Minimal Residual Method for Non-Hermitian Quaternion Linear Systems}{T. Li, Q.W. Wang, X.F. Zhang}

\title{QQMR: A Structure Preserving Quaternion Quasi-Minimal Residual Method for Non-Hermitian Quaternion Linear Systems\thanks{Corresponding author: Q.W. Wang (wqw@t.shu.edu.cn).
\funding{This work was funded by National Natural Science Foundation of China [grant numbers 12371023 and 12171369], Hainan Provincial Natural Science Foundation of China [grant numbers
122QN214 and 122MS001], and the Academic Programs project of Hainan University [grant number KYQD(ZR)-21119].}}}

\author{Tao Li\thanks{School of Mathematics and Statistics, Hainan University, Haikou 570228, P. R. China
  (\email{tli@hainanu.edu.cn}).}
\and Qing-Wen Wang\thanks{Department of Mathematics and Newtouch Center for Mathematics, and Collaborative Innovation Center for the Marine Artificial Intelligence, Shanghai University, Shanghai 200444, P. R. China
  (\email{wqw@t.shu.edu.cn}).}
\and Xin-Fang Zhang\thanks{School of Mathematics and Statistics, Hainan University, Haikou 570228, P. R. China
  (\email{995272@hainanu.edu.cn}).}}

\usepackage{amsopn}

\ifpdf
\hypersetup{
  pdftitle={ QQMR: A Structure-Preserving Quaternion Quasi-Minimal Residual Method for Non-Hermitian Quaternion Linear Systems},
  pdfauthor={Tao Li, Qing-Wen Wang}
}
\fi

\begin{document}

\maketitle

\begin{abstract}
 The quaternion biconjugate gradient (QBiCG) method, as a novel variant of quaternion Lanczos-type methods for solving the non-Hermitian quaternion linear systems, does not yield a minimization property. This means that the method possesses a rather irregular convergence behavior, which leads to numerical instability. In this paper, we propose a new structure-preserving quaternion quasi-minimal residual method, based on the quaternion biconjugate orthonormalization procedure with coupled two-term recurrences, which overcomes the drawback of QBiCG. The computational cost and storage required by the proposed method are much less than the traditional QMR iterations for the real representation of quaternion linear systems. Some convergence properties of which are also established. Finally, we report the numerical results to show the robustness and effectiveness of the proposed method compared with QBiCG.
\end{abstract}


\section{Introduction}\label{section-1}
Quaternions \cite{Hamilton}, invented by Sir W.R. Hamilton in 1843, and quaternion matrices \cite{Zhang1, Rodman, Qi} as powerful tools, have been widely used in orbital mechanics \cite{Jiang}, quantum mechanics \cite{Davies},  color image processing \cite{Subakan1, Jia0, Jia1, Chen0, Song1, Zhang11} and computer graphics \cite{Pletinckx, Guan1, Guan3}, etc.
The fundamental mathematical model among most of these areas can be formulated as the quaternion linear systems with quaternion vectors as desired variables, such as the Lorenz attractor model used in atmospheric turbulence \cite{Strogatz}, the color image pixel prediction model proposed in \cite{Jia4}, etc. All of these motivate us to consider high-performance algorithms for solving quaternion linear systems without dimension expanding.

Compared to the traditional Krylov subspace methods for solving the real or complex representations of quaternion linear systems, the structure-preserving quaternion Krylov subspace methods are far from abundant. The main difficulty comes from the non-commutability of multiplication for quaternions.
To overcome this issue, the traditional Krylov subspace methods are usually based on the real/complex representations, whose dimension is expanded to four/two times the original dimension. Obviously, the computational work and storage requirements of these methods are very costly. In contrast, the structure-preserving methods inheriting the algebraic symmetry of real representations do not need to expand the dimension. The latter can save three-quarters of the theoretical costs, which is more competitive. For more detailed implementations, see \cite{Jia2, Jia3, Li4, Ling1, Jia5, Jia6} and the references therein.

Let $\mathbb{R}$ be the real number field, and let $\mathbb{Q}$ be the quaternion skew-field of the form 
$$\mathbb{Q}=\{x_0+x_1\mathbf{i}+x_2\mathbf{j}+x_3\mathbf{k}|\mathbf{i}^2=\mathbf{j}^2=\mathbf{k}^2=\mathbf{i}\mathbf{j}\mathbf{k}=-1,x_0,x_1,x_2,x_3\in\mathbb{R}\}$$
where $\mathbf{i},\mathbf{j},\mathbf{k}$ are imaginary units. Denote by $\mathbb{Q}^{m\times n}$ and $\mathbb{Q}^{n}$, the collections of  all $m\times n$ matrices and $n$-dimensional vectors over $\mathbb{Q}$, respectively. In this paper, we consider the following non-Hermitian quaternion linear systems
\begin{equation}\label{1-1}
\mathbf{A}\mathbf{x}=\mathbf{b},
\end{equation}
where the invertible matrix $\mathbf{A}\in \mathbb{Q}^{n\times n}$ and the quaternion vector $\mathbf{b}\in \mathbb{Q}^{n}$ are given, $\mathbf{x}\in \mathbb{Q}^{n}$ is unknown required to be determined. Currently, only a few structure-preserving quaternion Krylov subspace methods and theoretical results for solving Eq.\eqref{1-1} have been elegantly developed. For example, Jia and Ng \cite{Jia4} derived the structure-preserving quaternion generalized minimal residual (QGMRES) method for solving the quaternion linear systems, which is superior to the block GMRES method \cite{Kubínová} over the complex field. Li and Wang \cite{Li0} formulated the quaternion full orthogonalization (QFOM) method and its preconditioned variant for obtaining the iterative solutions of Eq.\eqref{1-1}. Both are based on the quaternion Arnoldi procedure, which preserves the quaternion Hessenberg form while iterating. Besides, the quaternion biconjugate gradient (QBiCG) method \cite{Li1}, which relies on the quaternion Lanczos biorthogonalization procedure preserving the quaternion tridiagonal form while iterating, is proposed. The storage requirements of QBiCG, based on three-term recurrences, are lower than those of QGMRES and QFOM, such that it has better performance in some cases. However, the QBiCG iterates do not yield a minimization property in residual, which means that the method may often show an erratic convergence behavior with wild oscillations in the residual norm.

As well-known that the classical quasi-minimal residual (QMR) method
 \cite{Freund1, Freund2, Freund3, Freund4, Zhang} is an efficient and robust Krylov subspace method for solving general linear systems over the complex field. It can eliminate the numerical instability of the traditional 
 BiCG method. The iterates of the QMR method satisfy a quasi-minimization of the residual norm, rather than a Galerkin condition, which leads to a smooth and nearly monotone convergence
behavior. Moreover, in exact
arithmetic, it was shown in \cite{Freund2} that the QMR method developed by means of coupled two-term recursions is more robust than the equivalent three-term recurrences version. Hence, we propose a structure-preserving quaternion quasi-minimal residual (QQMR) method, which is based on the quaternion biconjugate orthonormalization process by means of coupled two-term recurrences, for solving the original linear systems \eqref{1-1}. The QQMR method depending on the stability of quaternion operations and the rapidity of real computations, only performs real operations without dimension expanding. The resulting tridiagonal quaternion least-squares problem is solved by the quaternion QR factorization with the generalized quaternion Givens rotations. Additionally, we will analyze the convergence for the proposed method and give some properties. Specifically, the method presented here differs from the one discussed in \cite{Freund1, Freund2} in the following aspects:
\begin{itemize}
    \item The proposed quaternion biconjugate orthonormalization (QBIO) process is entirely distinguished from the traditional two-sides Lanczos process \cite{Lanczos0}, since the QBIO process preserves the quaternion tridiagonal form during the iterations, but the latter cannot do it. The established process builds a pair of biconjugate orthonormal bases for the two quaternion subspaces
    \begin{equation}\label{1-2}
    \begin{aligned}
    \mathcal{K}_m(\mathbf{A},\mathbf{v}):&=\text{span}\{\mathbf{v},\mathbf{A}\mathbf{v},\cdots,\mathbf{A}^{m-1}\mathbf{v}\},\\
 \mathcal{K}_m(\mathbf{A}^*,\mathbf{w}):&=\text{span}\{\mathbf{w},\mathbf{A}^*\mathbf{w},\cdots,(\mathbf{A}^*)^{m-1}\mathbf{w}\}.\\
    \end{aligned}
    \end{equation}
    This differs from the quaternion Lanczos biorthogonalization procedure \cite{Li1}.
    \item Compared with the classical QMR method, the QQMR solver built upon the structure-preserving strategy significantly saves storage requirements and computational operations. Moreover, the developed method, as a high performance method over $\mathbb{Q}$, has better convergence properties than that of QBiCG, resulting in smoother convergence behavior.
\end{itemize}

The rest of this paper is as follows. In Section \ref{section-2}, we briefly recall some significant results associated with quaternions and quaternion matrices. In Section \ref{section-3}, we introduce the quaternion biconjugate orthonormalization procedure based on the three-term recurrences, which preserves the quaternion tridiagonal form during the iterations, and then develop a more robust version by means of the coupled two-term recursions. In Section \ref{section-4}, we develop two equivalently QQMR methods and their preconditioned variants, as well as present their convergence properties. Section \ref{section-5} is devoted to some numerical experiments with practical applications to illustrate the effectiveness and robustness of the QQMR method in comparison with QBiCG. Finally, we conclude this paper by giving some remarks in Section \ref{section-6}.



\section{Preliminaries}\label{section-2}
Throughout this paper, the vectors and matrices, unless otherwise stated, are assumed to be quaternions. For any quaternion $\mathbf{q}=q_0+q_1\mathbf{i}+q_2\mathbf{j}+q_3\mathbf{k}$, the conjugate of $\mathbf{x}$ is denoted by $\mathbf{q}^*=q_0-q_1\mathbf{i}-q_2\mathbf{j}-q_3\mathbf{k}$, and its magnitude is of the form $|\mathbf{q}|=\sqrt{\mathbf{q}^*\mathbf{q}}=\sqrt{q_{0}^2+q_{1}^2+q_{2}^2+q_{3}^2}\in \mathbb{R}$.
The inverse of nonzero quaternion $\mathbf{q}$, defined by $\mathbf{q}^{-1}=\mathbf{q}^*/{|\mathbf{q}|^2}$, is unique, and its conjugate is of the form $\mathbf{q}^{-*}=\mathbf{q}/{|\mathbf{q}|^2}$.  The inner product of two quaternion vectors $\mathbf{x}=(\mathbf{x}_i),\ \mathbf{y}=(\mathbf{y}_i)\in \mathbb{Q}^n$ is a quaternion scalar defined by
$$\langle\mathbf{x},\mathbf{y} \rangle:= \sum_{i=1}^n\mathbf{y}_i^*\mathbf{x}_i.$$
The induced 2-norm of the quaternion vector $\mathbf{x}$ is of the form
$$\|\mathbf{x}\|_2:=\sqrt{\sum_{i=1}^n|\mathbf{x}_i|^2}.$$
Then the 2-norm of the quaternion matrix $\mathbf{M}\in \mathbb{Q}^{m\times n}$ is defined as
$$
\|\mathbf{M}\|_2:=\max_{\mathbf{x}\in\mathbb{Q}^{n}/\{\mathbf{0}\}}\frac{\|\mathbf{M}\mathbf{x}\|_2}{\|\mathbf{x}\|_2}.
$$
 Denote by $\mathbf{v}_1\mathbf{\bm{\alpha}}_1+\mathbf{v}_2\mathbf{\bm{\alpha}}_2+\cdots+\mathbf{v}_n\mathbf{\bm{\alpha}}_n$, the right-hand-side linear combination of quaternion vectors $\mathbf{v}_1,\mathbf{v}_2,\cdots,\mathbf{v}_n$, in which ${\bm{\alpha}}_i$ are quaternion scalars ($i=1,2,\cdots,n$). If there exist quaternion scalars ${\bm{\alpha}}_1,\cdots,{\bm{\alpha}}_n$, not all zero, such that $\mathbf{v}_1\mathbf{\bm{\alpha}}_1+\mathbf{v}_2\mathbf{\bm{\alpha}}_2+\cdots+\mathbf{v}_n\mathbf{\bm{\alpha}}_n=\mathbf{0}$, then they are linearly dependent; in contrast, linearly independent. In what follows, we only consider the right multiplication over $\mathbb{Q}$ since it has similar properties to the unitary complex vector space. From \cite{Ghiloni}, one can see that $\mathbb{Q}^{n}$ is a right quaternionic Hilbert space with the inner product satisfying the following properties:
 \begin{itemize}
\item (Right linearity) $\langle\mathbf{x}\mathbf{\bm{\alpha}}+\mathbf{y}\mathbf{\bm{\beta}},\mathbf{z} \rangle=\langle\mathbf{x},\mathbf{z} \rangle\bm{\alpha}+\langle\mathbf{y},\mathbf{z} \rangle\bm{\beta}$ holds for $\mathbf{x},\mathbf{y},\mathbf{z}\in \mathbb{Q}^{n}$ and $\bm{\alpha},\bm{\beta}\in \mathbb{Q}$;
\item (Quaternionic hermiticity) $\langle\mathbf{x},\mathbf{y} \rangle=\langle\mathbf{y},\mathbf{x}\rangle^*$ holds for $\mathbf{x},\mathbf{y}\in \mathbb{Q}^{n}$;
\item (Positivity) $\langle\mathbf{x},\mathbf{x} \rangle\geq 0$ holds for $\mathbf{x} \in \mathbb{Q}^{n}$, and $\mathbf{x}=\mathbf{0}$ if and only if $\langle\mathbf{x},\mathbf{x} \rangle=0;$
\item If $\mathbf{x},\mathbf{y}\in \mathbb{Q}^{n}$, then their distance is defined as $d(\mathbf{x},\mathbf{y}):=\sqrt{\langle\mathbf{x}-\mathbf{y},\mathbf{x} -\mathbf{y}\rangle}$.
\end{itemize}
A quaternion matrix $\mathbf{X}\in\mathbb{Q}^{m\times n}$ is of the form $\mathbf{X}=X_{0}+X_{1}\mathbf{i}+X_{2}\mathbf{j}+X_{3}\mathbf{k}$, in which $X_{0}, X_{1}, X_{2}, X_{3}\in \mathbb{R}^{m\times n}$. Denote by $\mathbf{X}^*=X_{0}^{T}-X_{1}^{T}\mathbf{i}-X_{2}^{T}\mathbf{j}-X_{3}^{T}\mathbf{k}$,
 the conjugate transpose of $\mathbf{X}$, where the superscript $^T$ stands for the transpose of a real matrix. We say that the columns of a quaternion matrix  $\mathbf{V}$ are orthogonal to each other if $\mathbf{V}^*\mathbf{V}=\mathbf{I}$. 

Let $\mathbf{M}=M_0+M_1\mathbf{i}+M_2\mathbf{j}+M_3\mathbf{k}$ be an $m\times n$ quaternion matrix. From the linear homeomorphic mapping $\mathcal{R}(\cdot)$, the real counterpart of $\mathbf{M}$ is of the form
\begin{equation}\label{2-1}
\mathcal{R}(\mathbf{M})=\begin{bmatrix}
M_0&-M_1&-M_2&-M_3\\
M_1&M_0&-M_3&M_2\\
M_2&M_3&M_0&-M_1\\
M_3&-M_2&M_1&M_0\\
\end{bmatrix}\in \mathbb{R}^{4m\times 4n},
\end{equation}
and its first block column is denoted by
$
\mathcal{R}(\mathbf{M})_{c}=\begin{bmatrix}
M_0^T&M_1^T&M_2^T&M_3
\end{bmatrix}^T\in \mathbb{R}^{4m\times n}.
$
As stated in \cite{Jia3, Jia6}, one can see that the real counterpart of each quaternion matrix is surely a JRS-symmetric matrix. If $M, H, Q$ are JRS-symmetric matrices, then the matrix $\alpha_1 M+\alpha_2 HQ$ so is, where $\alpha_1, \alpha_2\in\mathbb{R}$. Denote the inverse mapping of $\mathcal{R}$ to the JRS-symmetric matrix by $\mathcal{R}^{-1}(\mathcal{R}(\mathbf{M}))=\mathbf{M}.$

The quaternion Krylov subspace generated by the coefficient matrix $\mathbf{A}\in \mathbb{Q}^{n\times n}$ and a nonzero vector $\mathbf{v}\in \mathbb{Q}^{n}$ is as follows
$$\mathcal{K}_m(\mathbf{A},\mathbf{v}):=\text{span}\{\mathbf{v},\mathbf{A}\mathbf{v},\cdots,\mathbf{A}^{m-1}\mathbf{v}\}$$
with $m\geq 1.$ Each vector in $\mathcal{K}_m(\mathbf{A},\mathbf{v})$ is a right-hand-side linear combination of $\mathbf{v},\mathbf{A}\mathbf{v},\cdots,\mathbf{A}^{m-1}\mathbf{v}$, but it cannot be rewritten as the multiplication of a polynomial with degree less than $m-1$ with the matrix 
$\mathbf{A}$ and the vector $\mathbf{v}$ because of the non-commutativity of the quaternion multiplication. For the sake of convenience, we denote the right-hand linear combination of $\mathbf{v},\mathbf{A}\mathbf{v},\cdots,\mathbf{A}^{m-1}\mathbf{v}$ by
\begin{equation}\label{2-2}
\mathcal{L}_m(\mathbf{A},\mathbf{v}):=\mathbf{v}\bm{\alpha}_0+\mathbf{A}\mathbf{v}\bm{\alpha}_1+\cdots+\mathbf{A}^{m-1}\mathbf{v}\bm{\alpha}_{m-1},
\end{equation}
where $\bm{\alpha}_0,\bm{\alpha}_1,\cdots,\bm{\alpha}_{m-1}\in\mathbb{Q}$. Obviously, the dimension of $\mathcal{K}_m(\mathbf{A},\mathbf{v})$ is the rank of the matrix $\mathbf{V}:=[\mathbf{v},\mathbf{Av},\cdots,\mathbf{A}^{m-1}\mathbf{v}]$. 

\section{Quaternion biconjugate orthonormalization method}\label{section-3}

In this section, we first formulate the quaternion biconjugate orthonormalization (QBIO) procedure employing three-term recursions, and then derive its equivalently robust version with coupled two-term recurrences.  Next, we recall some helpful results which will be used in the sequel. 

\begin{definition}\label{def-1}\cite{Li1} 
Let ${M}\in\mathbb{R}^{4n\times 4n}$ be a JRS-symmetric matrix having the form of \eqref{2-1}. Then it is said to be an upper JRS-tridiagonal matrix if $M_0\in\mathbb{R}^{n\times n}$ and $M_1, M_2,M_3 \in \mathbb{R}^{n\times n}$ are tridiagonal matrix and upper bidiagonal matrices, respectively. 
\end{definition}

\begin{theorem}\label{theorem-2}\cite{Li1} 
Let $M$ be an $4n\times 4n$ JRS-symmetric matrix having the form of \eqref{2-1}. Then there exist orthogonally JRS-symplectic matrices $V, W\in\mathbb{R}^{4n\times 4n}$, such that \begin{equation}\label{3-0}
W^TMV=H
\end{equation}
is a JRS-tridiagonal matrix.
\end{theorem}

\begin{theorem}\label{theorem-3}\cite{Li1} 
Let $\mathbf{A}=A_0+A_0\mathbf{i}+A_1\mathbf{j}+A_2\mathbf{k}$ be an $n\times n$ quaternion matrix, and let $m$ be a positive integer not greater than $n$. Then there exist two quaternion matrices $\mathbf{W},\mathbf{V}\in \mathbb{Q}^{n\times m}$ with unitary columns such that
\begin{equation}\label{2-3}
\mathbf{W}^*\mathbf{A}\mathbf{V}=\mathbf{H}
\end{equation}
is of tridiagonal form whose subdiagonal elements are nonnegative real numbers, or called a tridiagonal quaternion matrix.
\end{theorem}

Remarkably, we call that the decomposition in \eqref{3-0} is structure-preserving, since the matrix $H$ inherits the JRS-symmetry of the real counterpart $M$. The result also holds for different real counterparts of $\mathbf{M}$ because they are permutationally equivalent. Moreover, we only store the first block column of the real counterparts of quaternion matrices, not the explicit real counterparts, which saves storage and computational costs.

Given two initial nonzero quaternion vectors $\mathbf{v}_1$ and $\mathbf{w}_1$, the quaternion Lanczos biorthogonalization procedure \cite{Li1} for non-Hermitian quaternion matrices, see Algorithm \ref{algorithm-1}, produces two sequences of quaternion vectors $\mathbf{v}_1, \mathbf{v}_2,\cdots, \mathbf{v}_m $ and $\mathbf{w}_1, \mathbf{w}_2,\cdots, \mathbf{w}_m$, $m=1,2,\cdots$, such that 
    \begin{equation}\label{3-1}
    \begin{aligned}
    \mathcal{K}_m(\mathbf{A},\mathbf{v}_1):&=\text{span}\{\mathbf{v}_1, \mathbf{v}_2,\cdots, \mathbf{v}_m\},\\
 \mathcal{K}_m(\mathbf{A}^*,\mathbf{w}_1):&=\text{span}\{\mathbf{w}_1, \mathbf{w}_2,\cdots, \mathbf{w}_m\},\\
    \end{aligned}
    \end{equation}
 and 
\begin{equation}\label{3-2}
\langle\mathbf{v}_i,\mathbf{w}_j\rangle:=\left\{\begin{aligned}
&1\quad\quad \mbox{if}\quad i=j,\\
&0\quad\quad \mbox{if}\quad i\neq j.
\end{aligned}
\right.
\end{equation}
\begin{algorithm}[htb]
	\caption{ Quaternion Lanczos biorthogonalization procedure \cite{Li1}}\label{algorithm-1}
	\begin{algorithmic}[1]
		 \STATE  Choose two quaternion vectors $\mathbf{v}_1, \mathbf{w}_1$ such that $\langle\mathbf{v}_1,\mathbf{w}_1\rangle=1$. Set $\bm{\beta}_1=\bm{\sigma}_1=0,\mathbf{v}_0=\mathbf{w}_0=\mathbf{0}$.\\
		\STATE For $j=1,2,\cdots,m$ Do:\\
\STATE   $\bm{\alpha}_{j}=\langle \mathbf{A}\mathbf{v}_j,\mathbf{w}_j\rangle$\\
	\STATE $\overline{\mathbf{v}}_{j+1}:=\mathbf{A}\mathbf{v}_j-\mathbf{v}_{j}\bm{\alpha}_j-\mathbf{v}_{j-1}\bm{\beta}_j$\\
		\STATE $\overline{\mathbf{w}}_{j+1}:=\mathbf{A}^*\mathbf{w}_j-\mathbf{w}_{j}\bm{\alpha_j^*}-\mathbf{w}_{j-1}\bm{\sigma_j}$\\
\STATE $\bm{\sigma}_{j+1}=|\langle\overline{\mathbf{v}}_{j+1},\overline{\mathbf{w}}_{j+1}\rangle|^{\frac{1}{2}}$. If $\bm{\sigma}_{j+1}=0$ Stop\\
\STATE $\bm{\beta}_{j+1}=\langle\overline{\mathbf{v}}_{j+1},\overline{\mathbf{w}}_{j+1}\rangle\bm{\sigma}_{j+1}^{-1}$\\
		\STATE   $\mathbf{w}_{j+1}:=\overline{\mathbf{w}}_{j+1}\bm{\beta_{j+1}^{-*}}$\\
		\STATE $\mathbf{v}_{j+1}:=\overline{\mathbf{v}}_{j+1}\bm{\sigma}_{j+1}^{-1}$\\
	\end{algorithmic}
\end{algorithm}
Note that the coefficients $\bm{\sigma}_{j+1}$ and $\bm{\beta}_{j+1}$ in Algorithm \ref{algorithm-1} are scaling factors for the vectors $\mathbf{v}_{j+1}$ and $\mathbf{w}_{j+1}$,
and yield
\begin{equation}\label{3-3}
\bm{\sigma}_{j+1}\bm{\beta}_{j+1}:=\langle\overline{\mathbf{v}}_{j+1},\overline{\mathbf{w}}_{j+1}\rangle.
\end{equation}
However, both sequences $\{\mathbf{v}_{j+1}\}$ and  $\{\mathbf{w}_{j+1}\}$ may suffer from overflow or underflow, such that the approximate solution is far away from the exact solution. To prevent this drawback, in practical implementations, we follow the choices in the Lanczos process \cite{Lanczos0} that the unit quaternion vectors $\mathbf{v}_{j+1}$ and $\mathbf{w}_{j+1}$ are obtained by setting $\bm{\sigma}_{j+1}=\|\overline{\mathbf{v}}_{j+1}\|$ and $\bm{\beta}_{j+1}=\|\overline{\mathbf{w}}_{j+1}\|$, respectively, but they do not satisfy \eqref{3-3}. This gives the following algorithm, called the quaternion three-term biconjugate orthonormalization procedure, see Algorithm \ref{algorithm-2}.

\begin{algorithm}[htb]
	\caption{ Quaternion three-term biconjugate orthonormalization procedure}\label{algorithm-2}
	\begin{algorithmic}[1]
		 \STATE  Choose two quaternion vectors $\mathbf{v}_1, \mathbf{w}_1$ such that $\|\mathbf{v}_1\|=\|\mathbf{w}_1\|=1$, and $\bm{\sigma}_1=\langle\mathbf{v}_1,\mathbf{w}_1\rangle\neq 0$. Set $\mathbf{v}_0=\mathbf{w}_0=\mathbf{0}\in\mathbb{Q}^n$ and $\bm{\tau}_1=\bm{\rho}_1=0$, $\bm{\varepsilon}_1\neq 0$.\\
		\STATE For $j=1,2,\cdots,m$ Do:\\
\STATE   $\bm{\alpha}_{j}=\bm{\sigma}_j^{-1}\langle \mathbf{A}\mathbf{v}_j,\mathbf{w}_j\rangle$, $\overline{\bm{\alpha}}_{j}=\bm{\sigma}_j^{-*}\langle \mathbf{A}\mathbf{v}_j,\mathbf{w}_j\rangle^*$  \\
	\STATE $\overline{\mathbf{v}}_{j+1}:=\mathbf{A}\mathbf{v}_j-\mathbf{v}_{j}\bm{\alpha}_j-\mathbf{v}_{j-1}\bm{\tau}_j$\\
		\STATE $\overline{\mathbf{w}}_{j+1}:=\mathbf{A}^*\mathbf{w}_j-\mathbf{w}_{j}\overline{\bm{\alpha}}_{j}-\mathbf{w}_{j-1}(\bm{\rho}_j\bm{\sigma}_{j-1}^{-*}\bm{\sigma}_j^*)$\\
\STATE $\bm{\rho}_{j+1}=\|\overline{\mathbf{v}}_{j+1}\|$, and $\bm{\varepsilon}_{j+1}=\|\overline{\mathbf{w}}_{j+1}\|$\\
		\STATE   $\mathbf{w}_{j+1}:=\overline{\mathbf{w}}_{j+1}/\bm{\varepsilon}_{j+1}$\\
		\STATE $\mathbf{v}_{j+1}:=\overline{\mathbf{v}}_{j+1}/\bm{\rho}_{j+1}$\\
  \STATE $\bm{\sigma}_{j+1}=\langle\mathbf{v}_{j+1},\mathbf{w}_{j+1}\rangle$
  \STATE $\bm{\tau}_{j+1}=\bm{\varepsilon}_{j+1}\bm{\sigma}_{j}^{-1}\bm{\sigma}_{j+1}$
	\end{algorithmic}
\end{algorithm}

As seen, Algorithm \ref{algorithm-2} builds up a pair of biconjugate orthonormalized bases for the quaternion Krylov subspaces $\mathcal{K}_m(\mathbf{A},\mathbf{v}_1)$ and $\mathcal{K}_m(\mathbf{A}^*,\mathbf{w}_1)$, where the quaternion vectors $\mathbf{v}_j$ belongs to $\mathcal{K}_m(\mathbf{A},\mathbf{v}_1)$, while the $\mathbf{w}_j$'s are in the latter subspace. Moreover, using the real counterpart \eqref{2-1}, we can also see that the quaternion matrix-vector multiplications $\mathbf{A}\mathbf{v}_j$ and  $\mathbf{A}^*\mathbf{w}_j,$ and inner products $\langle \mathbf{A}\mathbf{v}_j,\mathbf{w}_{j}\rangle$ and $\langle {\mathbf{v}}_{j+1}, {\mathbf{w}}_{j+1}\rangle$ can be implemented by
\begin{equation}\label{3-33}
	\begin{aligned}
		\mathbf{A}\mathbf{v}_j&=\mathcal{R}^{-1}(\mathcal{R}(\mathbf{A})\mathcal{R}(\mathbf{v}_j)),\ \ \mathbf{A}^*\mathbf{w}_j=\mathcal{R}^{-1}(\mathcal{R}(\mathbf{A}^*)\mathcal{R}(\mathbf{w}_j)),\\
\langle \mathbf{A}\mathbf{v}_j,\mathbf{w}_{j}\rangle&=\mathcal{R}^{-1}(\mathcal{R}(\mathbf{w}_j^*)\mathcal{R}(\mathbf{A}\mathbf{v}_j)),\ \ \langle{\mathbf{v}}_{j+1},{\mathbf{w}}_{j+1}\rangle=\mathcal{R}^{-1}(\mathcal{R}({\mathbf{w}}_{j+1}^*)\mathcal{R}({\mathbf{v}}_{j+1})).
	\end{aligned}
\end{equation}
The crucial work of the above relations is how to fast compute  $\mathcal{R}(\mathbf{A})\mathcal{R}(\mathbf{v}_j)$, $\mathcal{R}(\mathbf{A}^*)$\\$\mathcal{R}(\mathbf{w}_j)$, $\mathcal{R}(\mathbf{w}_j^*)\mathcal{R}(\mathbf{A}\mathbf{v}_j)$ and $\mathcal{R}({\mathbf{w}}_{j+1}^*)\mathcal{R}({\mathbf{v}}_{j+1})$. From a practical point of view, this goal can be achieved by only computing the first block columns of the desired real counterparts, that is,
\begin{equation}\label{3-44}
	\begin{aligned}
\mathcal{R}(\mathbf{A}\mathbf{v}_j)_c&=\mathcal{R}(\mathbf{A})\mathcal{R}(\mathbf{v}_j)_c,\ \ \mathcal{R}(\mathbf{A}^*\mathbf{w}_j)_c=\mathcal{R}(\mathbf{A})^T\mathcal{R}(\mathbf{w}_j)_c,\\
		\mathcal{R}(\langle \mathbf{A}\mathbf{v}_j,\mathbf{w}_{j}\rangle)_c&=\mathcal{R}(\mathbf{w}_j)^T\mathcal{R}(\mathbf{A}\mathbf{v}_j)_c,\ \ \bm{\sigma}_{j+1}=|\mathcal{R}({\mathbf{w}}_{j+1})^T\mathcal{R}({\mathbf{v}}_{j+1})_c|^{\frac{1}{2}}.
	\end{aligned}
\end{equation}
This saves storage requirements and three-quarters of the theoretical costs. From the above results, we can deduce the following theorem.

\begin{theorem}\label{theorem-4}
    Let $\mathbf{V}_m:=[\mathbf{v}_1,\mathbf{v}_2,\cdots,\mathbf{v}_m]$, $\mathbf{W}_m:=[\mathbf{w}_1,\mathbf{w}_2,\cdots,\mathbf{w}_m]$ and 
    $$
    \mathbf{H}_m:=\begin{bmatrix}
\bm{\alpha}_1&\bm{\tau}_2&&&\\
\bm{\rho}_2&\bm{\alpha}_2&\bm{\tau}_3&&\\
&\ddots&\ddots&\ddots&\\
&&\bm{\rho}_{m-1}&\bm{\alpha}_{m-1}&\bm{\tau}_{m}\\
&&&\bm{\rho}_{m}&\bm{\alpha}_{m}\\
\end{bmatrix},\quad \mathbf{H}_{m+1,m}:=\begin{bmatrix}
\mathbf{H}_m\\
\bm{\rho}_{m+1}\mathbf{e}_m^*
\end{bmatrix},$$$$
\overline{\mathbf{H}}_m:=\begin{bmatrix}
\overline{\bm{\alpha}}_1&\overline{\bm{\tau}}_2&&&\\
\bm{\rho}_2&\overline{\bm{\alpha}}_2&\overline{\bm{\tau}}_3&&\\
&\ddots&\ddots&\ddots&\\
&&\bm{\rho}_{m-1}&\overline{\bm{\alpha}}_{m-1}&\overline{\bm{\tau}}_{m}\\
&&&\bm{\rho}_{m}&\overline{\bm{\alpha}}_{m}\\
\end{bmatrix},\quad \overline{\mathbf{H}}_{m+1,m}:=\begin{bmatrix}
\overline{\mathbf{H}}_m\\
\bm{\rho}_{m+1}\mathbf{e}_m^*
\end{bmatrix},
    $$
where $\overline{\bm{\tau}}_l=\bm{\varepsilon}_{l}\bm{\sigma}_{l-1}^{-*}\bm{\sigma}_{l}^*$, $l=2,\cdots,m$, $\bm{\rho}_{2},\cdots, \bm{\rho}_{m+1}$ and $\bm{\varepsilon}_{l}$ are positive real numbers.
     If Algorithm \ref{algorithm-2} proceeds $m$ steps successfully, then the quaternion vectors $\mathbf{v}_i$ and $\mathbf{w}_j$, $i,j=1,\cdots,m$, form a pair of biconjugate orthonormal bases, i.e.,
    \begin{equation}\label{3-4}
\langle\mathbf{v}_i,\mathbf{w}_j\rangle:=\left\{\begin{aligned}
&\bm{\sigma}_j\quad\quad \mbox{{\rm if}}\quad i=j,\\
&0\quad\quad\ \mbox{{\rm if}}\quad i\neq j.
\end{aligned}
\right.
\end{equation}
Moreover, $\{\mathbf{v}_i\}_{i=1,\cdots,m}$ is a basis of $\mathcal{K}_m(\mathbf{A},\mathbf{v}_1)$, while $\{\mathbf{w}_j\}_{j=1,\cdots,m}$ is a basis of $\mathcal{K}_m(\mathbf{A}^*,\mathbf{w}_1)$, and the following relations are true,
\begin{subequations}
\begin{align}
  &\mathbf{W}_m^*\mathbf{V}_m=\mathbf{D}_m,\label{eq:aa}\\
&\mathbf{A}\mathbf{V}_m=\mathbf{V}_{m+1}\mathbf{H}_{m+1,m},\label{eq:bb} \\
&\mathbf{A}^*\mathbf{W}_m=\mathbf{W}_{m+1}\Gamma_{m+1}^{-1}\overline{\mathbf{H}}_{m+1,m}\Gamma_{m}, \label{eq:cc}\\
&\mathbf{W}_m^*\mathbf{A}\mathbf{V}_m=\mathbf{D}_{m}\mathbf{H}_{m},\label{eq:dd} 
\end{align}
\end{subequations}
where $\mathbf{D}_m=diag(\bm{\sigma}_1,\bm{\sigma}_2,\cdots,\bm{\sigma}_m)$ with $\bm{\sigma}_i\neq 0$, $i=1,\cdots,m$, and $\Gamma_{m}=diag(\bm{s}_1,\bm{s}_2,$\\$\cdots,\bm{s}_m)$, $\bm{s}_i\in\mathbb{R}$, with
$$
\bm{s}_i:=\left\{\begin{aligned}
&1\quad\quad \quad\quad\quad\quad\mbox{\rm if}\quad i=1,\\
&\bm{s}_{i-1}\bm{\rho}_i/\bm{\varepsilon}_{i}\quad\quad\quad \mbox{\rm if}\quad i>1.
\end{aligned}
\right.
$$
\end{theorem}
\begin{proof}
We first prove the biconjugate orthonormality of quaternion vectors $\mathbf{v}_{k}$ and $\mathbf{w}_{k}$ by induction. Clearly, we obtain from Algorithm \ref{algorithm-2} that $\mathbf{v}_{k}$ and $\mathbf{w}_{k}$ are unit vectors. This means that we only need to show the biorthogonality of the above vectors. When $k=1$, it follows that $\langle\mathbf{v}_1,\mathbf{w}_1\rangle=\bm{\sigma}_1$. Assume now that the assertion holds for the case $1\leq k\leq j$. For $k=j+1$, we show that
$\langle\mathbf{v}_{j+1},\mathbf{w}_i\rangle=0$ holds for $i\leq j$. If $i=j$, then it follows that
$$\begin{aligned}
\langle\mathbf{v}_{j+1},\mathbf{w}_j\rangle&=\bm{\rho}_{j+1}^{-1}[\langle\mathbf{A}\mathbf{v}_j,\mathbf{w}_j\rangle-\langle\mathbf{v}_{j},\mathbf{w}_j\rangle\bm{\alpha}_j-\langle\mathbf{v}_{j-1},\mathbf{w}_j\rangle\bm{\tau}_j]\\
&=\bm{\rho}_{j+1}^{-1}[\langle\mathbf{A}\mathbf{v}_j,\mathbf{w}_j\rangle-\langle\mathbf{v}_{j},\mathbf{w}_j\rangle\bm{\sigma}_j^{-1}\langle \mathbf{A}\mathbf{v}_j,\mathbf{w}_j\rangle]\\
&=0.
\end{aligned}
$$
Consider now the inner product $\langle\mathbf{v}_{j+1},\mathbf{w}_i\rangle$ with $i<j$, we have 
$$\begin{aligned}
\langle\mathbf{v}_{j+1},\mathbf{w}_i\rangle&=\bm{\rho}_{j+1}^{-1}[\langle\mathbf{v}_j,\mathbf{A}^*\mathbf{w}_i\rangle-\langle\mathbf{v}_{j},\mathbf{w}_i\rangle\bm{\alpha}_j-\langle\mathbf{v}_{j-1},\mathbf{w}_i\rangle\bm{\tau}_j]\\
&=\bm{\rho}_{j+1}^{-1}[\langle\mathbf{v}_j,\mathbf{w}_{i+1}\bm{\varepsilon}_{i+1}+\mathbf{w}_{i}\overline{\bm{\alpha}}_{i}+\mathbf{w}_{i-1}(\bm{\rho}_i\bm{\sigma}_{i-1}^{-*}\bm{\sigma}_j^*)\rangle-\langle\mathbf{v}_{j-1},\mathbf{w}_i\rangle\bm{\tau}_j]
\\
&=\bm{\rho}_{j+1}^{-1}[\langle\mathbf{v}_j,\mathbf{w}_{i+1}\rangle\bm{\varepsilon}_{i+1}-\langle\mathbf{v}_{j-1},\mathbf{w}_i\rangle\bm{\tau}_j]
\end{aligned}$$
For $i<j-1$, it is easy to see that the induction hypothesis vanishes the inner products in the above expression. For $i=j-1$, it follows that 
$$\begin{aligned}
\langle\mathbf{v}_{j+1},\mathbf{w}_{j-1}\rangle&=\bm{\rho}_{j+1}^{-1}[\langle\mathbf{v}_j,\mathbf{w}_{j}\rangle\bm{\varepsilon}_{j}-\langle\mathbf{v}_{j-1},\mathbf{w}_{j-1}\rangle\bm{\tau}_j]\\
&=\bm{\rho}_{j+1}^{-1}[\bm{\sigma}_{j}\bm{\varepsilon}_{j}-\bm{\varepsilon}_{j}\bm{\sigma}_{j-1}\bm{\sigma}_{j-1}^{-1}\bm{\sigma}_{j}]\\
&=0.
\end{aligned}$$
Likewise, we can prove that $\langle\mathbf{v}_i,\mathbf{w}_{j+1}\rangle=0$ holds for $i\leq j$. The proof of the matrix relations \eqref{eq:bb}-\eqref{eq:cc} is similar to those presented in \cite{Saad, Saad1}, hence we omit it. By Theorem \ref{theorem-3}, the remaining relation \eqref{eq:dd} is surely true.
\end{proof}
\begin{remark}\label{remark-1}
Note that the matrix relations \eqref{eq:bb} and \eqref{eq:cc} are distinct from those given in \cite{Saad, Saad1}, since the matrix $\overline{\mathbf{H}}_{m+1,m}$ is not equal to ${\mathbf{H}}_{m+1,m}$, even though both of them are tridiagonal quaternion matrices. It is also worth underlying that Algorithm \ref{algorithm-2} reduces to the classical biconjugate orthonormalization process if all its elements are complex/real numbers.
\end{remark}

From Algorithm \ref{algorithm-2}, we can see that a pair of biconjugate orthonormalized bases for $\mathcal{K}_m(\mathbf{A},\mathbf{v}_1)$ and $\mathcal{K}_m(\mathbf{A}^*,\mathbf{w}_1)$ was deduced by three-term recurrences. However, as stated by \cite{Freund2, Lanczos1, Gutknecht0}, the coupled two-term recurrences for achieving the same goal, which is based on a suitable second set of basis vectors for underlying quaternion Krylov subspace, perform more accurate and stable than the former. This motivates us to generate a pair of biconjugate orthonormalized bases based on the desired coupled two-term recurrences. Moreover, to exclude overflow or underflow, we scale the obtained quaternion basis vectors to unit length. Using the quaternion LU decomposition \cite{Wang, Li}, we now assume that the matrices ${\mathbf{H}}_{m+1,m}$ and $\overline{\mathbf{H}}_{m+1,m}$ can be decomposed into the following forms
\begin{equation}\label{3-5}
{\mathbf{H}}_{m+1,m}={\mathbf{L}}_{m+1,m}{\mathbf{U}}_{m},\quad\quad \overline{\mathbf{H}}_{m+1,m}=\overline{\mathbf{L}}_{m+1,m}\overline{\mathbf{U}}_{m},
\end{equation}
where $$\begin{aligned}
{\mathbf{L}}_{m+1,m}&=\begin{bmatrix}
\bm{\tau}_1^{(1)}&&&\\
\bm{\rho}_2&\bm{\tau}_2^{(1)}&&\\
&\ddots&\ddots&\\
&&\bm{\rho}_{m}&\bm{\tau}_{m}^{(1)}\\
&&&\bm{\rho}_{m+1}\\
\end{bmatrix},\quad {\mathbf{U}}_{m}=\begin{bmatrix}
1&\mathbf{u}_2^{(1)}&&&\\
&1&\mathbf{u}_3^{(1)}&&\\
&&\ddots&\ddots&\\
&&&1&\mathbf{u}_m^{(1)}\\
&&&&1\\
\end{bmatrix},\\
\overline{\mathbf{L}}_{m+1,m}&=\begin{bmatrix}
\bm{\tau}_1^{(2)}&&&\\
\bm{\rho}_2&\bm{\tau}_2^{(2)}&&\\
&\ddots&\ddots&\\
&&\bm{\rho}_{m}&\bm{\tau}_{m}^{(2)}\\
&&&\bm{\rho}_{m+1}\\
\end{bmatrix},\quad \overline{\mathbf{U}}_{m}=\begin{bmatrix}
1&\mathbf{u}_2^{(2)}&&&\\
&1&\mathbf{u}_3^{(2)}&&\\
&&\ddots&\ddots&\\
&&&1&\mathbf{u}_m^{(2)}\\
&&&&1\\
\end{bmatrix}
\end{aligned}$$
are bidiagonal quaternion matrices with $\bm{\tau}_{k}^{(r)}, \mathbf{u}_s^{(r)}\in \mathbb{Q}$, $k=1,\cdots,m, s=2,\cdots,m,$ $ r=1,2$, and $\bm{\rho}_{2},\cdots,\bm{\rho}_{m+1}$ being positive real numbers. Let
$$
\mathbf{P}_m=[\mathbf{p}_1\cdots,\mathbf{p}_m]=\mathbf{V}_m\mathbf{U}_m^{-1},\quad\quad \mathbf{Q}_m=[\mathbf{q}_1\cdots,\mathbf{q}_m]=\mathbf{W}_m\Gamma_{m}^{-1}\overline{\mathbf{U}}_{m}^{-1}\Gamma_{m}.
$$
Next, we will show that the coupled two-term recurrences are generated by the matrices $\mathbf{P}_m$ and $\mathbf{Q}_m$. By Eqs.\eqref{eq:aa}-\eqref{eq:dd} and the above results, it follows that
\begin{subequations}
\begin{align}
\mathbf{V}_m&=\mathbf{P}_m\mathbf{U}_m,\label{eq:aa1}\\
\mathbf{W}_m&=\mathbf{Q}_m\Gamma_{m}^{-1}\overline{\mathbf{U}}_{m}\Gamma_{m},\label{eq:bb1}\\
\mathbf{A}\mathbf{P}_m&=\mathbf{V}_{m+1}\mathbf{L}_{m+1,m},\label{eq:cc1}\\
\mathbf{A}^*\mathbf{Q}_m&=\mathbf{W}_{m+1}\Gamma_{m+1}^{-1}\overline{\mathbf{L}}_{m+1,m}\Gamma_{m},\label{eq:dd1}
\end{align}
\end{subequations}
and we thus have the following relations
\begin{subequations}
\begin{align}
\mathbf{v}_k&=\mathbf{p}_k+\mathbf{p}_{k-1}\bm{\mu}_k^{(1)},\label{eq:aa2}\\
\mathbf{w}_k&=\mathbf{q}_k+\mathbf{q}_{k-1}(\bm{s}_k/\bm{s}_{k-1})\bm{\mu}_k^{(2)},\label{eq:bb2}\\
\mathbf{A}\mathbf{p}_k&=\mathbf{v}_{k}\bm{\tau}_{k}^{(1)}+\mathbf{v}_{k+1}\bm{\rho}_{k+1},\label{eq:cc2}\\
\mathbf{A}^*\mathbf{q}_k&=\mathbf{w}_{k}\bm{\tau}_{k}^{(2)}+\mathbf{w}_{k+1}(\bm{s}_{k}/\bm{s}_{k+1})\bm{\rho}_{k+1},\label{eq:dd2}
\end{align}
\end{subequations}
where $\mathbf{p}_0=\mathbf{q}_0=\bf{0}\in \mathbb{Q}^n,$ $k=1,\cdots,m$.
From the line 5 of Algorithm \ref{algorithm-2} and $\bm{s}_{k+1}=\bm{s}_{k}\bm{\rho}_{k+1}/\bm{\varepsilon}_{k+1}$, the mentioned-above relations can be rewritten as
\begin{subequations}
\begin{align}
\mathbf{p}_k&=\mathbf{v}_k-\mathbf{p}_{k-1}\bm{\mu}_k^{(1)},\label{eq:aa3}\\
\mathbf{q}_k&=\mathbf{w}_k-\mathbf{q}_{k-1}(\bm{s}_k/\bm{s}_{k-1})\bm{\mu}_k^{(2)},\label{eq:bb3}\\
\overline{\mathbf{v}}_{k+1}&=\mathbf{A}\mathbf{p}_k-\mathbf{v}_{k}\bm{\tau}_{k}^{(1)},\label{eq:cc3}\\
\overline{\mathbf{w}}_{k+1}&=\mathbf{A}^*\mathbf{q}_k-\mathbf{w}_{k}\bm{\tau}_{k}^{(2)},\label{eq:dd3}
\end{align}
\end{subequations}
where $\overline{\mathbf{v}}_{k+1}=\mathbf{v}_{k+1}\bm{\rho}_{k+1}$ and $\overline{\mathbf{w}}_{k+1}=\mathbf{w}_{k+1}\bm{\varepsilon}_{k+1}$. Inspired by the derivation in \cite{Freund2,Bücker}, here we expect that the quaternion vectors $\mathbf{p}_k$ and $\mathbf{q}_k$ to the matrices $\mathbf{P}_m$ and $\mathbf{Q}_m$ are $\mathbf{A}$-biorthogonal, i.e.,
 \begin{equation}\label{3-6}
\langle\mathbf{A}\mathbf{p}_k,\mathbf{q}_j\rangle:=\left\{\begin{aligned}
&\bm{l}_k\neq 0\quad\quad \mbox{{\rm if}}\quad k=j,\\
&0\quad\quad\quad\quad\ \mbox{{\rm if}}\quad k\neq j.
\end{aligned}
\right.
\end{equation}
 Then we present the details for computing the parameters $\bm{\mu}_k^{(1)}, \bm{\mu}_k^{(2)}, \bm{\tau}_{k}^{(1)}, $ and $ \bm{\tau}_{k}^{(1)}$. 
 Taking the right-hand inner product on both sides of \eqref{eq:cc3} with $\mathbf{w}_k$, it follows that
 $$\begin{aligned}
0&=\langle\mathbf{A}\mathbf{p}_k-\mathbf{v}_k\bm{\tau}_{k}^{(1)},\mathbf{w}_k\rangle\\
&=\langle\mathbf{A}\mathbf{p}_k,\mathbf{w}_k\rangle-  \langle\mathbf{v}_k,\mathbf{w}_k\rangle\bm{\tau}_{k}^{(1)}\\
&=\langle\mathbf{A}\mathbf{p}_k,\mathbf{q}_k+\mathbf{q}_{k-1}(\bm{s}_k/\bm{s}_{k-1})\bm{\mu}_k^{(2)}\rangle-  \langle\mathbf{v}_k,\mathbf{w}_k\rangle\bm{\tau}_{k}^{(1)}\\
&=\langle\mathbf{A}\mathbf{p}_k,\mathbf{q}_k\rangle-  \langle\mathbf{v}_k,\mathbf{w}_k\rangle\bm{\tau}_{k}^{(1)},\\
 \end{aligned}$$
which shows that 
$\bm{\tau}_{k}^{(1)}=\bm{\sigma}_{k}^{-1}\bm{l}_k.$ Similarly, for \eqref{eq:dd3}, we have
$$\begin{aligned}
0&=\langle\mathbf{A}^*\mathbf{q}_k-\mathbf{w}_k\bm{\tau}_{k}^{(2)},\mathbf{v}_k\rangle\\
&=\langle\mathbf{A}^*\mathbf{q}_k,\mathbf{v}_k\rangle-  \langle\mathbf{w}_k,\mathbf{v}_k\rangle\bm{\tau}_{k}^{(2)}\\
&=\langle\mathbf{A}^*\mathbf{q}_k,\mathbf{p}_k+\mathbf{p}_{k-1}\bm{\mu}_k^{(1)}\rangle-  \langle\mathbf{w}_k,\mathbf{v}_k\rangle\bm{\tau}_{k}^{(2)}\\
&=\langle\mathbf{A}^*\mathbf{q}_k,\mathbf{p}_k\rangle-  \langle\mathbf{w}_k,\mathbf{v}_k\rangle\bm{\tau}_{k}^{(2)},\\
 \end{aligned}$$
that is, $\bm{\tau}_{k}^{(2)}=\bm{\sigma}_{k}^{-*}\bm{l}_k^*$.  Taking the right-hand inner product on both sides of \eqref{eq:aa3} with $\mathbf{A}^*\mathbf{q}_{k-1}$, it follows from Eq.\eqref{3-4} that
$$\begin{aligned}
0&=\langle\mathbf{v}_k-\mathbf{p}_{k-1}\bm{\mu}_k^{(1)},\mathbf{A}^*\mathbf{q}_{k-1}\rangle\\
&=\langle\mathbf{v}_k,\mathbf{w}_{k}\bm{\varepsilon}_{k}+\mathbf{w}_{k-1}\bm{\tau}_{k-1}^{(2)}\rangle-\langle\mathbf{p}_{k-1},\mathbf{A}^*\mathbf{q}_{k-1}\rangle\bm{\mu}_k^{(1)}\\
&=\langle\mathbf{v}_k,\mathbf{w}_{k}\rangle\bm{\varepsilon}_{k}-\langle\mathbf{A}\mathbf{p}_{k-1},\mathbf{q}_{k-1}\rangle\bm{\mu}_k^{(1)},
\end{aligned}
$$
that is, $\bm{\mu}_k^{(1)}=\bm{\varepsilon}_{k}\bm{l}_{k-1}^{-1}\bm{\sigma}_{k}$. Moreover, we can deduce that the parameter $\bm{\mu}_k^{(2)}=\bm{\varepsilon}_{k}\bm{l}_{k-1}^{-*}\bm{\sigma}_{k}^*$ in a similar way.

In what follows, we conversely prove that the obtained quaternion vectors $\mathbf{p}_k$ and $\mathbf{q}_k$ form a $\mathbf{A}$-biorthogonal system having the form of Eq.\eqref{3-6}. Denote by $\mathbf{L}_k$ the matrix obtained from $\mathbf{L}_{k+1,k}$ by deleting its last row. From Eqs.\eqref{eq:dd} and \eqref{3-5}, it follows that
$$\begin{aligned}
\mathbf{Q}_k^*\mathbf{A}\mathbf{P}_k&=\Gamma_{k}^{-*}\overline{\mathbf{U}}_{k}^{-*}\Gamma_{k}^*\mathbf{W}_k^*\mathbf{A}\mathbf{V}_k\mathbf{U}_k^{-1}\\
&=\Gamma_{k}^{-*}\overline{\mathbf{U}}_{k}^{-*}\Gamma_{k}^*\mathbf{D}_k\mathbf{H}_k\mathbf{U}_k^{-1}\\
&=\Gamma_{k}^{-*}\overline{\mathbf{U}}_{k}^{-*}\Gamma_{k}^*\mathbf{D}_k\mathbf{L}_k\\
\end{aligned}$$
is of the form
$$
\begin{bmatrix}
\bm{\sigma}_{1}\bm{\tau}_1^{(1)}&&&\\
-\frac{\bm{s}_2}{\bm{s}_1}(\bm{\mu}_{2}^{(2)})^*\bm{\sigma}_1\bm{\tau}_1^{(1)}+\bm{\sigma}_2\bm{\rho}_2&\bm{\sigma}_{2}\bm{\tau}_2^{(1)}&&\\
&\ddots&\ddots&\\
&&-\frac{\bm{s}_{k}}{\bm{s}_{k-1}}(\bm{\mu}_{k}^{(2)})^*\bm{\sigma}_{k-1}\bm{\tau}_{k-1}^{(1)}+\bm{\sigma}_k\bm{\rho}_k&\bm{\sigma}_{k}\bm{\tau}_{k}^{(1)}\\
\end{bmatrix}.
$$
Substitute $\bm{\tau}_k^{(1)}$ and $\bm{\mu}_{k}^{(2)}$ into the above matrix, we can readily obtain that 
$$\mathbf{Q}_k^*\mathbf{A}\mathbf{P}_k:=\text{diag}(\bm{l}_1,\cdots,\bm{l}_k).$$
This shows that the scenario is surely true. According to the above discussion, the quaternion coupled two-term biconjugate orthonormalization procedure, which is mathematically equivalent to Algorithm \ref{algorithm-2} in exact arithmetic, can be described as follows.
\begin{algorithm}[htb]
	\caption{ Quaternion coupled two-term biconjugate orthonormalization procedure}\label{algorithm-3}
	\begin{algorithmic}[1]
		 \STATE  Choose two quaternion vectors $\mathbf{v}_1, \mathbf{w}_1$ such that $\|\mathbf{v}_1\|=\|\mathbf{w}_1\|=1$, and $\bm{\sigma}_1=\langle\mathbf{v}_1,\mathbf{w}_1\rangle\neq 0$. Set $\mathbf{p}_0=\mathbf{q}_0=\mathbf{0}\in\mathbb{Q}^n$ and $\bm{l}_0=0$.\\
		\STATE For $j=1,2,\cdots,m$ Do:\\
\STATE $\mathbf{p}_j:=\mathbf{v}_j-\mathbf{p}_{j-1}\bm{\varepsilon}_{j}\bm{l}_{j-1}^{-1}\bm{\sigma}_{j}$
\STATE $\mathbf{q}_j:=\mathbf{w}_j-\mathbf{q}_{j-1}\bm{\rho}_j\bm{l}_{j-1}^{-*}\bm{\sigma}_{j}^*$
\STATE $\bm{l}_j=\langle\mathbf{A}\mathbf{p}_j,\mathbf{q}_j\rangle$
	\STATE $\overline{\mathbf{v}}_{j+1}:=\mathbf{A}\mathbf{p}_j-\mathbf{v}_{j}\bm{\sigma}_{j}^{-1}\bm{l}_j$\\
		\STATE $\overline{\mathbf{w}}_{j+1}:=\mathbf{A}^*\mathbf{q}_j-\mathbf{w}_{j}\bm{\sigma}_{j}^{-*}\bm{l}_j^*$\\
\STATE $\bm{\rho}_{j+1}=\|\overline{\mathbf{v}}_{j+1}\|$, and $\bm{\varepsilon}_{j+1}=\|\overline{\mathbf{w}}_{j+1}\|$\\
		\STATE   $\mathbf{w}_{j+1}:=\overline{\mathbf{w}}_{j+1}/\bm{\varepsilon}_{j+1}$\\
		\STATE $\mathbf{v}_{j+1}:=\overline{\mathbf{v}}_{j+1}/\bm{\rho}_{j+1}$\\
	\STATE $\bm{\sigma}_{j+1}=\langle\mathbf{v}_{j+1},\mathbf{w}_{j+1}\rangle$
 \end{algorithmic}
\end{algorithm}

Similar to the classical QMR method, a major concern of Algorithm \ref{algorithm-3} occurs the exact breakdowns or near-breakdowns in case of $\bm{\sigma}_{j}=0$ or $\bm{l}_j=0$. The exact breakdowns, also called lucky breakdowns, are unlikely to be encountered in practice. However, the near-breakdowns will occur when $\bm{\sigma}_{j}$ or $\bm{l}_j$ are nonzero, but small in some sense. Fortunately, there are remedies for this problem that
allow the algorithm to continue in most cases and improve its numerical stability. One of the most popular remedies is the look-ahead techniques \cite{Freund1, Freund2, Parlett} that skip over possible near-breakdowns, but its drawback is the nonnegligible added complexity.  Indeed, the near-breakdowns in linear systems are rare, and their effect is generally benign.
Therefore, we think that a simpler remedy, such as restarting the biconjugate orthonormalization procedure, may well be adequate here.
\begin{algorithm}[htb]
	\caption{ Restarted Algorithm \ref{algorithm-3}}\label{algorithm-4}
	\begin{algorithmic}[1]
		 \STATE  Choose two quaternion vectors $\mathbf{v}_1, \mathbf{w}_1$ such that $\|\mathbf{v}_1\|=\|\mathbf{w}_1\|=1$, and $\bm{\sigma}_1=\langle\mathbf{v}_1,\mathbf{w}_1\rangle\neq 0$. Set $\mathbf{p}_0=\mathbf{q}_0=\mathbf{0}\in\mathbb{Q}^n$ and $\bm{l}_0=0$.\\
		\STATE  Generate a pair of biorthonormal bases and the quaternion matrices $\mathbf{L}_{m+1,m}$ and $\overline{\mathbf{L}}_{m+1,m}$ from Algorithm \ref{algorithm-3}.
	\STATE If satisfied then Stop, otherwise, the near-breakdowns occur, then set $\mathbf{v}_{1}=\mathbf{v}_{m},$ $\mathbf{w}_{1}=\mathbf{w}_{m}$ and go to step 1.
 \end{algorithmic}
\end{algorithm}

Again, the quaternion matrix-vector products and inner products, such as $\mathbf{A}\mathbf{p}_j, \bm{l}_j$, $\mathbf{A}^*\mathbf{q}_j$ and $\bm{\sigma}_j$, are theoretically computed by the real structure-preserving methods shown in Eq.\eqref{3-44}. We generate these products once their first block columns are obtained, saving arithmetic operations.

\section{Quaternion quasi-minimal residual method}\label{section-4}
In this section, we develop two different quaternion quasi-minimal residual (QQMR) methods for solving the quaternion linear systems \eqref{1-1}. Let $\mathbf{x}_0$ be an initial guess to the original linear systems. Then, the QQMR method is started with
 $\mathbf{w}_1=\mathbf{v}_1=\mathbf{r}_0/\bm{\beta}\neq 0$, in which $\mathbf{r}_0=\mathbf{b}-\mathbf{A}\mathbf{x}_0$ and $\bm{\beta}=\|\mathbf{r}_0\|_2$. An approximate solution $\mathbf{x}_m$ of QQMR at $m$-th iterate extracted from the affine subspace $\mathbf{x}_0+\mathcal{K}_m(\mathbf{A},\mathbf{v}_1)$ has the following form \begin{equation}\label{4-1}
\mathbf{x}_m:=\mathbf{x}_0+\mathbf{V}_m\mathbf{z}_m,
\end{equation}
where $\mathbf{V}_m\in\mathbf{Q}^{n\times m}$ is generated by either Algorithm \ref{algorithm-2} or Algorithm \ref{algorithm-3}, $\mathbf{z}_m\in\mathbb{Q}^m$ is determined by a quasi-minimal residual property.
 It follows from Eq.\eqref{eq:bb} that the corresponding residual vector is given by
\begin{equation}\label{4-2}
\begin{aligned}
\mathbf{r}_m&=\mathbf{b}_0-\mathbf{A}(\mathbf{x}_0+\mathbf{V}_m\mathbf{z}_m)\\
&=\mathbf{r}_0-\mathbf{V}_{m+1}\mathbf{H}_{m+1,m}\mathbf{z}_m,
\end{aligned}
\end{equation}
where $\mathbf{e}_1=[1,0,\cdots,0]\in\mathbb{R}^{n+1}$. 

The first QQMR method is based on the quaternion three-term recurrences, i.e., formulated by Algorithm \ref{algorithm-2}. Clearly, the residual norm obtained from \eqref{4-2}  is of the form
$$\|\mathbf{r}_m\|_2=\|\mathbf{V}_{m+1}(\bm{\beta}\mathbf{e}_1-\mathbf{H}_{m+1,m}\mathbf{z}_m)\|_2.$$
Notice that it is still a reasonable idea to
minimize the quasi-residual norm
\begin{equation}\label{4-3}
\|\bm{\beta}\mathbf{e}_1-\mathbf{H}_{m+1,m}\mathbf{z}_m\|_2=\min_{\mathbf{z}\in\mathbb{Q}^m}\|\bm{\beta}\mathbf{e}_1-\mathbf{H}_{m+1,m}\mathbf{z}\|_2
\end{equation}
and compute the approximate solution, even though the column-vectors of $\mathbf{V}_{m+1}$ are not orthogonal to each other. We note that all subdiagonal elements of $\mathbf{H}_{m+1,m}$ shown in Theorem \ref{theorem-4} are nonnegative real numbers and thus it has full rank, which means that the minimizer $\mathbf{z}_m$ of problem \eqref{4-3} is unique. Moreover, by taking advantage of the structure of $\mathbf{H}_{m+1,m}$, we can transform the tridiagonal quaternion matrix $\mathbf{H}_{m+1,m}$ into upper tridiagonal form by a series of generalized quaternion Givens rotations \cite{Jia4, Jia2}.  Assume now that the elements $\bm{\rho}_{i+1}\neq 0$ and $\mathbf{t}_{i}=\|\left[\begin{array}{ccccc}
	\bm{\alpha}_{i}&
	\bm{\rho}_{i+1}\\
\end{array}\right]^T\|_2\neq 0$ hold for $i=1,\cdots,m,$ then the $i$th generalized quaternion Givens rotation defined in \cite{Jia2, Jia3} is an ${(m+1)\times m}$ quaternion matrix having the form
$$\mathbf{G}_i=\left[\begin{array}{ccccc}
	\mathbf{I}_{i-1}&&&\\
	&\mathbf{g}_{11}^{(i)}&\mathbf{g}_{12}^{(i)}&\\
	&\mathbf{g}_{21}^{(i)}&\mathbf{g}_{22}^{(i)}&\\
	&&&\mathbf{I}_{m-i}\\
\end{array}\right]$$
in which
$$
\mathbf{g}_{11}^{(i)}=\frac{\bm{\alpha}_{i}}{\mathbf{t}_{i}},\ \mathbf{g}_{21}^{(i)}=\frac{\bm{\rho}_{i+1}}{\mathbf{t}_{i}},
$$
$$\left\{
 \begin{aligned}
 	\mathbf{g}_{12}^{(i)}&=|\mathbf{g}_{21}^{(i)}|,\  \mathbf{g}_{22}^{(i)}=-\frac{\mathbf{g}_{21}^{(i)}}{|\mathbf{g}_{21}^{(i)}|}(\mathbf{g}_{11}^{(i)})^*,\ \ \mbox{if}\  |\bm{\alpha}_{i}|\leq |\bm{\rho}_{i+1}|,\\
\mathbf{g}_{22}^{(i)}&=|\mathbf{g}_{11}^{(i)}|,\  \mathbf{g}_{12}^{(i)}=-\frac{\mathbf{g}_{11}^{(i)}}{|\mathbf{g}_{11}^{(i)}|}(\mathbf{g}_{21}^{(i)})^*,\ \ \mbox{if}\ |\bm{\alpha}_{i}|>|\bm{\rho}_{i+1}|.\\
 		\end{aligned}
 \right.$$
 Multiply the Hessenberg matrix $\mathbf{H}_{m+1,m}$ and the corresponding right-hand side $\mathbf{g}_m=\bm{\beta}\mathbf{e}_1$ by $\mathbf{G}_1^*,\cdots,\mathbf{G}_m^*$ from the left, i.e., the QR factorization of $\mathbf{H}_{m+1,m}$. Denote the product of $\mathbf{G}_i^*$ as the unitary quaternion $\mathbf{Q}_m $, i.e., $\mathbf{Q}_m=\mathbf{G}_m^*\cdots\mathbf{G}_1^*$. Then the problem \eqref{4-2} can be reduced as an upper tridiagonal quaternion least-squares (TQLS) problem
\begin{equation}\label{4-4}
\mathbf{z}_m=\min_{\mathbf{z}\in\mathbb{Q}^m}\|\mathbf{g}_{m+1}-\mathbf{R}_{m+1,m}\mathbf{z}\|_2,
\end{equation}
where $\mathbf{g}_{m+1}=\left[\begin{array}{ccccc}
	\mathbf{g}_{m}\\
	\bm{\gamma}_{m+1}\\
\end{array}\right]=\mathbf{Q}_m(\bm{\beta}\mathbf{e}_1)$, $\mathbf{R}_{m+1,m}=\left[\begin{array}{ccccc}
	\mathbf{R}_{m}\\
	0\\
\end{array}\right]=\mathbf{Q}_m\mathbf{H}_{m+1,m}$, with
\begin{equation}\label{4-5}\mathbf{R}_{m}=\begin{bmatrix}
\bm{\eta}_{1}^{(1)}&\bm{\eta}_{1}^{(2)}&\bm{\eta}_{1}^{(3)}&&&\\
&\bm{\eta}_{2}^{(1)}&\bm{\eta}_{2}^{(2)}&\bm{\eta}_{2}^{(3)}&&\\
&&\ddots&\ddots&\ddots&\\
&&&\bm{\eta}_{m-2}^{(1)}&\bm{\eta}_{m-2}^{(2)}&\bm{\eta}_{m-2}^{(3)}\\
&&&&\bm{\eta}_{m-1}^{(1)}&\bm{\eta}_{m-1}^{(2)}\\
&&&&&\bm{\eta}_{m}^{(1)}\\
\end{bmatrix},\quad\quad
\mathbf{g}_{m}=\left[\begin{array}{ccccc}
	\bm{\gamma}_1\\
	\vdots\\
 \bm{\gamma}_{m}\\
\end{array}\right],\end{equation}
$\bm{\eta}_{i}^{(1)}$ being positive real numbers, $\bm{\eta}_{j}^{(2)},$ and $ \bm{\eta}_{k}^{(3)}\in\mathbb{Q}$, $i=1,\cdots,m, j=1,\cdots,m-1, k=1,\cdots,m-2$. It seems that the storage and computational work of \eqref{4-3} are much less than \eqref{4-2}. Overall, we can obtain that the unique solution to the TQLS problem is given by
\begin{equation}\label{4-6}
 \mathbf{z}_m:=\mathbf{R}_{m}^{-1}\mathbf{g}_m,
\end{equation}
and the approximate solution is thus of form
\begin{equation}\label{4-7}
\mathbf{x}_m:=\mathbf{x}_0+\mathbf{V}_m\mathbf{R}_{m}^{-1}\mathbf{g}_m.
\end{equation}
Then, its corresponding quasi-residual norm is of the form 
\begin{equation}\label{4-8}
\begin{aligned}
\|\bm{\beta}\mathbf{e}_1-\mathbf{H}_{m+1,m}\mathbf{z}_m\|_2
=|\bm{\gamma}_{m+1}|.
\end{aligned}
\end{equation}

Actually, the solution $\mathbf{x}_m$ can be readily updated from step to step. The details are as follows. Define
$\mathbf{D}^{(m)}=\left[\begin{array}{ccccc}
	\mathbf{d}_1&\mathbf{d}_2&\cdots&\mathbf{d}_m
\end{array}\right]\equiv \mathbf{V}_m\mathbf{R}_{m}^{-1}$, then it follows that
$$\begin{aligned}
\mathbf{d}_1&=\mathbf{v}_1/\bm{\eta}_{1}^{(1)},\\
\mathbf{d}_2&=(\mathbf{v}_2-\mathbf{d}_1\bm{\eta}_{1}^{(2)})/\bm{\eta}_{2}^{(1)},\\
&\cdots\\
\mathbf{d}_i&=(\mathbf{v}_i-\mathbf{d}_{i-1}\bm{\eta}_{i-1}^{(2)}-\mathbf{d}_{i-2}\bm{\eta}_{i-2}^{(3)})/\bm{\eta}_{i}^{(1)},\quad\quad i=3,4,\cdots.
\end{aligned}$$
Then, Eq.\eqref{4-7} can be rewritten as
$$\begin{aligned}
\mathbf{x}_m&=\mathbf{x}_0+\mathbf{V}_m\mathbf{R}_{m}^{-1}\mathbf{g}_m\\
&=\mathbf{x}_0+\mathbf{D}^{(m)}\mathbf{g}_m\\
&=\mathbf{x}_0+\left[\begin{array}{ccccc}
	\mathbf{D}^{(m-1)}&\mathbf{d}_m
\end{array}\right]\left[\begin{array}{ccccc}
	\mathbf{g}_{m-1}\\
 \bm{\gamma}_m
\end{array}\right]\\
&=\mathbf{x}_0+\mathbf{D}^{(m-1)}\mathbf{g}_{m-1}+\mathbf{d}_m \bm{\gamma}_m\\
&=\mathbf{x}_{m-1}+\mathbf{d}_m \bm{\gamma}_m.
\end{aligned}
$$
Utilizing this information, a QQMR method derived from Algorithm \ref{algorithm-2} is as follows.
 \begin{algorithm}[htb]
	\caption{QQMR method based on quaternion three-term recurrences}\label{algorithm-5}
	\begin{algorithmic}[1]
		 \STATE  Given an initial guess $\mathbf{x}_0$ and the tolerance $\varepsilon$. Compute $\mathbf{r}_0=\mathbf{b}-\mathbf{A}\mathbf{x}_0$ and $\bm{\beta}=\bm{\gamma}_1=\|\mathbf{r}_0\|_2$. Set  $\mathbf{w}_1:=\mathbf{v}_1:=\mathbf{r}_0/\bm{\gamma}_1$.
		\STATE For $j=1,2,\cdots,$ until convergence Do:\\
\STATE   $\bm{\alpha}_{j}=\bm{\sigma}_j^{-1}\langle \mathbf{A}\mathbf{v}_j,\mathbf{w}_j\rangle$, $\overline{\bm{\alpha}}_{j}=\bm{\sigma}_j^{-*}\langle \mathbf{A}\mathbf{v}_j,\mathbf{w}_j\rangle^*$  \\
	\STATE $\overline{\mathbf{v}}_{j+1}:=\mathbf{A}\mathbf{v}_j-\mathbf{v}_{j}\bm{\alpha}_j-\mathbf{v}_{j-1}\bm{\tau}_j$\\
 \STATE $\overline{\mathbf{w}}_{j+1}:=\mathbf{A}^*\mathbf{w}_j-\mathbf{w}_{j}\overline{\bm{\alpha}}_{j}-\mathbf{w}_{j-1}(\bm{\rho}_j\bm{\sigma}_{j-1}^{-*}\bm{\sigma}_j^*)$\\
 \STATE If $j>2$ then
$$\left[\begin{array}{ccccc}
	\bm{\eta}_{j-2}^{(3)}\\
	\overline{\bm{\tau}}_{j}\\
\end{array}\right]=\left[\begin{array}{ccccc}
	\mathbf{g}_{11}^{(j-2)}&\mathbf{g}_{12}^{(j-2)}\\
	\mathbf{g}_{21}^{(j-2)}&\mathbf{g}_{22}^{(j-2)}\\
\end{array}\right]^*\left[\begin{array}{ccccc}
	0\\
	\bm{\tau}_{j}\\
\end{array}\right]$$
end if
 \STATE If $j>1$ then
 $$\left[\begin{array}{ccccc}
	\bm{\eta}_{j-1}^{(2)}\\
	\hat{\bm{\alpha}}_{j}\\
\end{array}\right]=\left[\begin{array}{ccccc}
	\mathbf{g}_{11}^{(j-1)}&\mathbf{g}_{12}^{(j-1)}\\
	\mathbf{g}_{21}^{(j-1)}&\mathbf{g}_{22}^{(j-1)}\\
\end{array}\right]^*\left[\begin{array}{ccccc}
	\overline{\bm{\tau}}_{j}\\
	\bm{\alpha}_{j}\\
\end{array}\right],\quad\mbox{where}\ \overline{\bm{\tau}}_{2}={\bm{\tau}}_{2}$$
end if
 \STATE If $|\hat{\bm{\alpha}}_{j}|\leq |\bm{\rho}_{j+1}|$, $(\hat{\bm{\alpha}}_{1}={\bm{\alpha}}_{1})$, then
\STATE $\mathbf{g}_{12}^{(j)}=|\mathbf{g}_{21}^{(j)}|,\  \mathbf{g}_{22}^{(j)}=-\frac{\mathbf{g}_{21}^{(j)}}{|\mathbf{g}_{21}^{(j)}|}(\mathbf{g}_{11}^{(j)})^*$,\ \mbox{where} $\mathbf{g}_{11}^{(j)}=\frac{\hat{\bm{\alpha}}_{j}}{\mathbf{t}_{j}},\ \mathbf{g}_{21}^{(j)}=\frac{\bm{\rho}_{j+1}}{\mathbf{t}_{j}},$ 
 \STATE Else
 \STATE
$\mathbf{g}_{22}^{(j)}=|\mathbf{g}_{11}^{(j)}|,\  \mathbf{g}_{12}^{(j)}=-\frac{\mathbf{g}_{11}^{(j)}}{|\mathbf{g}_{11}^{(j)}|}(\mathbf{g}_{21}^{(j)})^*$,\ \mbox{where} $\mathbf{g}_{11}^{(j)}=\frac{\hat{\bm{\alpha}}_{j}}{\mathbf{t}_{j}},\ \mathbf{g}_{21}^{(j)}=\frac{\bm{\rho}_{j+1}}{\mathbf{t}_{j}},$ 
\STATE end if
 \STATE $\bm{\eta}_{j}^{(1)}=(\mathbf{g}_{11}^{(j)})^*\hat{\bm{\alpha}}_{j}+(\mathbf{g}_{21}^{(j)})^*\bm{\rho}_{j+1}$
 \STATE $$\left[\begin{array}{ccccc}
	{\bm{\gamma}}_{j}\\
	{\bm{\gamma}}_{j+1}\\
\end{array}\right]=\left[\begin{array}{ccccc}
	\mathbf{g}_{11}^{(j)}&\mathbf{g}_{12}^{(j)}\\
	\mathbf{g}_{21}^{(j)}&\mathbf{g}_{22}^{(j)}\\
\end{array}\right]^*\left[\begin{array}{ccccc}
	{\bm{\gamma}}_{j}\\
	0\\
\end{array}\right]$$
 \STATE $\mathbf{d}_j=(\mathbf{v}_j-\mathbf{d}_{j-1}\bm{\eta}_{j-1}^{(2)}-\mathbf{d}_{j-2}\bm{\eta}_{j-2}^{(3)})/\bm{\eta}_{j}^{(1)}$ \mbox{where} $\mathbf{d}_0=\mathbf{d}_{-1}=\mathbf{0}$
 \STATE $\mathbf{x}_{j}=\mathbf{x}_{j-1}+\mathbf{d}_j \bm{\gamma}_j$, $\mathbf{r}_j=\mathbf{r}_{j-1}-\mathbf{A}\mathbf{d}_j\bm{\gamma}_j$
 \STATE If $\|\mathbf{r}_j\|/{\bm{\beta}}\leq \varepsilon$, then Stop	
\STATE $\bm{\rho}_{j+1}=\|\overline{\mathbf{v}}_{j+1}\|$, and $\bm{\varepsilon}_{j+1}=\|\overline{\mathbf{w}}_{j+1}\|$\\
		\STATE   $\mathbf{w}_{j+1}:=\overline{\mathbf{w}}_{j+1}/\bm{\varepsilon}_{j+1}$\\
		\STATE $\mathbf{v}_{j+1}:=\overline{\mathbf{v}}_{j+1}/\bm{\rho}_{j+1}$\\
  \STATE $\bm{\sigma}_{j+1}=\langle\mathbf{v}_{j+1},\mathbf{w}_{j+1}\rangle$
  \STATE $\bm{\tau}_{j+1}=\bm{\varepsilon}_{j+1}\bm{\sigma}_{j}^{-1}\bm{\sigma}_{j+1}$
	\end{algorithmic}
\end{algorithm}

The following proposition presents a result on the actual residual norm of the solution of Algorithm \ref{algorithm-5}.
\begin{proposition}\label{proposition-4-1}
Let $\mathbf{x}_m$ be an approximate solution generated by Algorithm \ref{algorithm-5} at step $m$. Then its residual norm yields the relation
$$\|\mathbf{r}_m\|_2\leq \sqrt{m+1}|(\mathbf{g}_{12}^{(1)})^*(\mathbf{g}_{12}^{(2)})^*\cdots(\mathbf{g}_{12}^{(m)})^*|\|\mathbf{r}_0\|_2.$$
\begin{proof}
Since the $m+1$ columns of $\mathbf{V}_{m+1}$ are unit vectors, this implies $\|\mathbf{V}_{m+1}\|_2\leq \sqrt{m+1}$.
Moreover, by line 14 in Algorithm \ref{algorithm-5}, we obtain
$$|\bm{\gamma}_{m+1}|\leq |(\mathbf{g}_{12}^{(1)})^*(\mathbf{g}_{12}^{(2)})^*\cdots(\mathbf{g}_{12}^{(m)})^*|\|\mathbf{r}_0\|_2.$$
The assertion following Eqs.\eqref{4-2} and \eqref{4-8} immediately yields
$$\begin{aligned}
\|\mathbf{r}_m\|_2&\leq\|\mathbf{V}_{m+1}\|_2|\bm{\gamma}_{m+1}|\\
&\leq \|\mathbf{r}_m\|_2\leq \sqrt{m+1}|(\mathbf{g}_{12}^{(1)})^*(\mathbf{g}_{12}^{(2)})^*\cdots(\mathbf{g}_{12}^{(m)})^*|\|\mathbf{r}_0\|_2.
\end{aligned}$$
\end{proof}
\end{proposition}

Next, we derive an upper bound
on the convergence rate for Algorithm \ref{algorithm-5}, which is essentially the same as the bound for QGMRES \cite{Jia4}. Ideally, in the absence of round-off errors, the Algorithm \ref{algorithm-5} will terminate at most $4n$ iterations, just like QBiCG \cite{Li1}. Given a diagonalizable quaternion matrix
$\mathbf{M}$, the condition number of a quaternion matrix $\mathbf{X}$, consisted of eigenvectors of $\mathbf{M}$, is defined as
$$
\bm{\kappa}(\mathbf{M})=\min_{\mathbf{X}:\mathbf{X}^{-1}\mathbf{M}\mathbf{X}\ \mbox{diagonal}} \|\mathbf{X}^{-1}\|_2\|\mathbf{X}\|_2.
$$ With this information, the convergence theorem of Algorithm \ref{algorithm-5} can be formulated as follows.
\begin{theorem}\label{theorem-5}
Assume that the tridiagonal quaternion matrix $\mathbf{H}_{m}\in\mathbb{Q}^{m\times m}$ generated by Algorithm \ref{algorithm-5} is diagonalizable, and let $\mathbf{H}_m:=\mathbf{X}\mathbf{S}\mathbf{X}^{-1}$ with $\mathbf{S}=diag(\bm{\lambda}_1,\cdots,\bm{\lambda}_m)$ whose diagonal elements are eigenvalues of $\mathbf{H}_m$. Then, for $k=1, \cdots, m-1$, the residual vectors of QQMR method yield
\begin{equation}\label{4-9}
\|\mathbf{r}_k\|_2\leq \|\mathbf{r}_0\|_2\sqrt{k+1}\bm{\kappa}(\mathbf{H}_m)\bm{\zeta}^{(k)},
\end{equation}
where $\bm{\zeta}^{(k)}=\min_{\bm{\alpha}_1,\cdots,\bm{\alpha}_{j}\in\mathbb{Q}, j\leq k}(\max_{i=1,\cdots,m}(1+|\bm{\lambda}_i||\bm{\alpha}_1|+\cdots+|\bm{\lambda}_i^j||\bm{\alpha}_j|)).$
Moreover, if the Algorithm \ref{algorithm-5} breaks down at step $m$, i.e., $\bm{\rho}_{m+1}=0$, then $\mathbf{x}_m=\mathbf{A}^{-1}\mathbf{b}$ is the exact solution of the original linear systems \eqref{1-1}.
\begin{proof}
Let \begin{equation}\label{4-10}
\bm{\chi}_k=\min_{\mathbf{z}\in \mathbb{Q}^{k}}\|\mathbf{e}_1^{(k+1)}-\mathbf{H}_{k+1,k}\mathbf{z}\|_2,\quad\quad  \mathbf{e}_1^{(k+1)}=[0,\cdots,0,1]^T\in\mathbb{R}^{k+1},
\end{equation}
and let
$k\in\{1,2,\cdots,m-1\}$ be arbitrary, but fixed. By
$$\mathbf{H}_m=\left[\begin{array}{ccccc}
	\mathbf{H}_k&*\\
0&*
\end{array}\right],$$ it follows that
$$\mathbf{H}_m\left[\begin{array}{ccccc}
	\mathbf{z}\\
0
\end{array}\right]=\left[\begin{array}{ccccc}
	\mathbf{H}_k\mathbf{z}\\
0
\end{array}\right]$$
for all $\mathbf{z}\in\mathbb{Q}^{k}$. Recall that the quaternion matrix $\mathbf{H}_m$ is tridiagonal with nonnegative subdiagonal elements and Eq.\eqref{2-2}, we have
$$
\left\{\left[\begin{array}{ccccc}
	\mathbf{z}\\
0
\end{array}\right]|\mathbf{z}\in\mathbb{Q}^k\right\}=\{\mathcal{L}_{k-1}(\mathbf{H}_m,\mathbf{e}_1^{(m)})\}.
$$
Then Eq.\eqref{4-10} can be rewritten as
$$\bm{\chi}_k=\min_{\mathbf{z}\in \mathbb{Q}^{k}}\|\mathbf{e}_1^{(m)}-\mathbf{H}_{m}\left[\begin{array}{ccccc}
	\mathbf{z}\\
0
\end{array}\right]\|_2=\min\|\mathcal{L}_{k}(\mathbf{H}_m,\mathbf{e}_1^{(m)})\|_2,$$
where $\mathcal{L}_{k}(\mathbf{H}_m,\mathbf{e}_1^{(m)})$ satisfy the constraint $\mathcal{L}_{k}(\mathbf{0},\mathbf{e}_1^{(m)})=\mathbf{e}_1^{(m)} (\bm{\alpha}_0=1)$. By expanding $\mathbf{e}_1^{(m)} $ into any set of eigenvectors of $\mathbf{H}_m$, the above relation leads to
$$\begin{aligned}
\bm{\chi}_k&=\min\|\mathcal{L}_{k}(\mathbf{H}_m,\mathbf{e}_1^{(m)})\|_2\\
&\leq \min\|\mathbf{I}+|\mathbf{H}_m||\bm{\alpha}_1|+\cdots+|\mathbf{H}_m^k||\bm{\alpha}_k|\|_2\|\mathbf{e}_1^{(m)}\|_2\\
&\leq \|\mathbf{X}\|_2 \|\mathbf{X}^{-1}\|_2(\min\|\mathbf{I}+|\mathbf{S}||\bm{\alpha}_1|+\cdots+|\mathbf{S}^k||\bm{\alpha}_k|\|_2)\\
&\leq \bm{\kappa}(\mathbf{H}_m)(\min_{\bm{\alpha}_1,\cdots,\bm{\alpha}_{j}\in\mathbb{Q}, j\leq k}(\max_{i=1,\cdots,m}(1+|\bm{\lambda}_i||\bm{\alpha}_1|+\cdots+|\bm{\lambda}_i^j||\bm{\alpha}_j|))).
\end{aligned}$$
This together with Proposition \ref{proposition-4-1} give the following relation
$$
\|\mathbf{r}_k\|_2\leq\sqrt{k+1}\|\mathbf{r}_k\|_2\bm{\chi}_k
\leq\|\mathbf{r}_0\|_2\sqrt{k+1}\bm{\zeta}^{(k)},
$$
which completes the first half of the proof. When $k=m$, the quaternion least-square problem \eqref{4-3} reduces to a quaternion linear system with the coefficient matrix $\mathbf{H}_m$. Since $\mathbf{A}$ is nonsingular, then $\mathbf{H}_m$ so is, and thus it can be solved exactly. The proof follows immediately.
\end{proof}
\end{theorem}

In what follows, we derive a stable variant of the QQMR method based on the quaternion coupled two-term recurrence, i.e., Algorithm \ref{algorithm-3}, for solving the quaternion linear system \eqref{1-1}. Recall that Algorithm \ref{algorithm-3}, Eqs.\eqref{3-5}, \eqref{eq:aa1}  and \eqref{4-2}, we observe that the approximate solution $\mathbf{x}_m$ of \eqref{1-1} can be reformulated as follows
\begin{equation}\label{4-11}
    \mathbf{x}_m=\mathbf{x}_0+\mathbf{P}_{m}\mathbf{y}_m,
\end{equation}
where $\mathbf{y}_m=\mathbf{U}_{m}\mathbf{z}_m$ is the unique solution of the following lower bidiagonal quaternion least-square problem
\begin{equation}\label{4-12}
    \|\bm{\beta}\mathbf{e}_1-\mathbf{L}_{m+1,m}\mathbf{y}_m\|_2=\min_{\mathbf{y}\in\mathbb{Q}^m}\|\bm{\beta}\mathbf{e}_1-\mathbf{L}_{m+1,m}\mathbf{y}\|_2.
\end{equation}
Again, by a sequence of generalized quaternion Givens rotations, the QR factorization of $\mathbf{L}_{m+1,m}$ is of the form
$\mathbf{Q}_m\mathbf{L}_{m+1,m}=\left[\begin{array}{ccccc}
	\overline{\mathbf{R}}_{m}\\
	0\\
\end{array}\right]$, with
$$\overline{\mathbf{R}}_{m}=\begin{bmatrix}
\bm{\kappa}_{1}^{(1)}&\bm{\kappa}_{1}^{(2)}&0&&\\
&\bm{\kappa}_{2}^{(1)}&\bm{\kappa}_{2}^{(2)}&0&\\
&&\ddots&\ddots&0\\
&&&\bm{\kappa}_{m-1}^{(1)}&\bm{\kappa}_{m-1}^{(2)}\\
&&&&\bm{\kappa}_{m}^{(1)}\\
\end{bmatrix},$$
and  $\overline{\mathbf{g}}_{m+1}=\left[\begin{array}{ccccc}
	\overline{\mathbf{g}}_{m}\\
	\overline{\bm{\gamma}}_{m+1}\\
\end{array}\right]=\mathbf{Q}_m(\bm{\beta}\mathbf{e}_1)$. Then, similar to Algorithm \ref{algorithm-5}, the solution $\mathbf{y}_m$ is given by $ \mathbf{y}_m:=\overline{\mathbf{R}}_{m}^{-1}\overline{\mathbf{g}}_m.$ Define
$\overline{\mathbf{D}}^{(m)}=\left[\begin{array}{ccccc}
\overline{\mathbf{d}}_1&\overline{\mathbf{d}}_2&\cdots&\overline{\mathbf{d}}_m
\end{array}\right]\equiv \mathbf{P}_m\overline{\mathbf{R}}_{m}^{-1}$, then we can update $\mathbf{x}_m$ from step to step in a similar way. The details are as follows. By
$$\begin{aligned}
\overline{\mathbf{d}}_1&=\mathbf{p}_1/\bm{\kappa}_{1}^{(1)},\\
&\cdots\\
\overline{\mathbf{d}}_i&=(\mathbf{p}_i-\overline{\mathbf{d}}_{i-1}\bm{\kappa}_{i-1}^{(2)})/\bm{\kappa}_{i}^{(1)}, \quad\quad i=2,3,\cdots.
\end{aligned}$$

With these results, a robust version of the QQMR method based on Algorithm \ref{algorithm-3} is described below, see Algorithm \ref{algorithm-6}.

\begin{algorithm}[htb]
	\caption{QQMR method based on quaternion coupled two-term recurrences}\label{algorithm-6}
	\begin{algorithmic}[1]
		 \STATE  Given an initial guess $\mathbf{x}_0$ and the tolerance $\varepsilon$. Compute $\mathbf{r}_0=\mathbf{b}-\mathbf{A}\mathbf{x}_0$ and $\bm{\beta}=\bm{\gamma}_1=\|\mathbf{r}_0\|_2$. Set  $\mathbf{w}_1:=\mathbf{v}_1:=\mathbf{r}_0/\bm{\gamma}_1$, $\mathbf{p}_0=\mathbf{q}_0=\mathbf{0}\in\mathbb{Q}^n$ and $\bm{l}_0=0$.\\
		\STATE For $j=1,2,\cdots,$ until convergence Do:\\
\STATE $\mathbf{p}_j:=\mathbf{v}_j-\mathbf{p}_{j-1}\bm{\varepsilon}_{j}\bm{l}_{j-1}^{-1}\bm{\sigma}_{j}$
\STATE $\mathbf{q}_j:=\mathbf{w}_j-\mathbf{q}_{j-1}\bm{\rho}_j\bm{l}_{j-1}^{-*}\bm{\sigma}_{j}^*$
\STATE $\bm{l}_j=\langle\mathbf{A}\mathbf{p}_j,\mathbf{q}_j\rangle$
\STATE $\overline{\mathbf{v}}_{j+1}:=\mathbf{A}\mathbf{p}_j-\mathbf{v}_{j}\bm{\sigma}_{j}^{-1}\bm{l}_j$\\
		\STATE $\overline{\mathbf{w}}_{j+1}:=\mathbf{A}^*\mathbf{q}_j-\mathbf{w}_{j}\bm{\sigma}_{j}^{-*}\bm{l}_j^*$\\
 \STATE Update the QR factorization of $\mathbf{L}_{j+1,j}$, i.e., Generate $\overline{\mathbf{R}}_{j}$ and $\overline{\mathbf{g}}_{j}$ by applying a sequence of $\mathbf{G}_i^*$.
 \STATE $\overline{\mathbf{d}}_j=(\mathbf{p}_j-\overline{\mathbf{d}}_{j-1}\bm{\kappa}_{j-1}^{(2)})/\bm{\kappa}_{j}^{(1)} \mbox{where}\ \overline{\mathbf{d}}_0=\mathbf{0}$
 \STATE $\mathbf{x}_{j}=\mathbf{x}_{j-1}+\overline{\mathbf{d}}_j\overline{\bm{\gamma}}_j$, $\mathbf{r}_j=\mathbf{r}_{j-1}-\mathbf{A}\overline{\mathbf{d}}_j\overline{\bm{\gamma}}_j$
 \STATE If $\|\mathbf{r}_j\|/{\bm{\beta}}\leq \varepsilon$, then Stop	
\STATE $\bm{\rho}_{j+1}=\|\overline{\mathbf{v}}_{j+1}\|$, and $\bm{\varepsilon}_{j+1}=\|\overline{\mathbf{w}}_{j+1}\|$\\
		\STATE   $\mathbf{w}_{j+1}:=\overline{\mathbf{w}}_{j+1}/\bm{\varepsilon}_{j+1}$\\
		\STATE $\mathbf{v}_{j+1}:=\overline{\mathbf{v}}_{j+1}/\bm{\rho}_{j+1}$\\
	\STATE
 $\bm{\sigma}_{j+1}=\langle\mathbf{v}_{j+1},\mathbf{w}_{j+1}\rangle$
\end{algorithmic}
\end{algorithm}
Note that Algorithm \ref{algorithm-5} and Algorithm \ref{algorithm-6} are equivalent, which means that all convergence properties for the former also hold for the latter algorithm.

It should be pointed out that preconditioning is a crucial ingredient for the success of quaternion Krylov subspace methods in practical applications. In general, the reliability and validity of the above methods, dealing with various applications, depends more on the quality of the preconditioner. Herein, we take the nonsingular quaternion matrix $\mathbf{M}=\mathbf{M}_1\mathbf{M}_2\in\mathbb{Q}^{n\times n}$ as a preconditioner, which approximates in some sense the coefficient matrix $\mathbf{A}$ of Eq.\eqref{1-1}. Then we apply the QQMR Algorithms \ref{algorithm-5} and \ref{algorithm-6} to solve the preconditioned quaternion linear system
$$\mathbf{A}'\mathbf{x}'=\mathbf{b}'$$
which has the same solution with Eq.\eqref{1-1} and better spectral properties, where $\mathbf{A}'=\mathbf{M}_1^{-1}\mathbf{A}\mathbf{M}_2^{-1}$, $\mathbf{x}'=\mathbf{M}_2\mathbf{x}$ and $\mathbf{b}'=\mathbf{M}_1^{-1}\mathbf{b}$. Hence, the preconditioned QQMR methods, including preconditioned Algorithms \ref{algorithm-5} and \ref{algorithm-6}, are sketched as follows.

\begin{algorithm}[htb]
	\caption{Preconditioned QQMR method based on quaternion three-term recurrences}\label{algorithm-7}
	\begin{algorithmic}[1]
		 \STATE  Given an initial guess $\mathbf{x}_0$ and the tolerance $\varepsilon$. Compute $\mathbf{r}_0^{'} =\mathbf{M}_1^{-1}(\mathbf{b}-\mathbf{A}\mathbf{x}_0)$ and $\bm{\beta}=\bm{\gamma}_1=\|\mathbf{r}_0^{'}\|_2$. Set  $\mathbf{w}_1:=\mathbf{v}_1:=\mathbf{r}_0^{'}/\bm{\gamma}_1$.
		\STATE For $j=1,2,\cdots,$ until convergence Do:\\
\STATE  Perform the $j$-th iteration of Algorithm \ref{algorithm-5} (applied to $\mathbf{A}'$). The generated quaternion matrices $\mathbf{V}_j, \mathbf{V}_{j+1}$ and $\mathbf{H}_{j+1,j}$ yield
$\mathbf{A}'\mathbf{V}_j=\mathbf{V}_{j+1}\mathbf{H}_{j+1,j}$.
 \STATE $\mathbf{x}_{j}=\mathbf{x}_{j-1}+\mathbf{d}_j \bm{\gamma}_j$, $\mathbf{r}_j=\mathbf{r}_{j-1}-\mathbf{A}^{'}\mathbf{d}_j\bm{\gamma}_j$
 \STATE If $\|\mathbf{r}_j\|/{\bm{\beta}}\leq \varepsilon$, then Stop	
	\end{algorithmic}
\end{algorithm}
\begin{algorithm}[htb]
	\caption{Preconditioned QQMR method based on quaternion coupled two-term recurrences}\label{algorithm-8}
	\begin{algorithmic}[1]
		 \STATE  Given an initial guess $\mathbf{x}_0$ and the tolerance $\varepsilon$. Compute $\mathbf{r}_0^{'} =\mathbf{M}_1^{-1}(\mathbf{b}-\mathbf{A}\mathbf{x}_0)$ and $\bm{\beta}=\bm{\gamma}_1=\|\mathbf{r}_0^{'}\|_2$. Set  $\mathbf{w}_1:=\mathbf{v}_1:=\mathbf{r}_0^{'}/\bm{\gamma}_1$, $\mathbf{p}_0=\mathbf{q}_0=\mathbf{0}\in\mathbb{Q}^n$ and $\bm{l}_0=0$.
		\STATE For $j=1,2,\cdots,$ until convergence Do:\\
\STATE  Perform the $j$-th iteration of Algorithm \ref{algorithm-6} (applied to $\mathbf{A}'$). The generated quaternion matrices $\mathbf{P}_j, \mathbf{V}_{j+1}$ and $\mathbf{L}_{j+1,j}$ yield
$\mathbf{A}^{'}\mathbf{P}_j=\mathbf{V}_{j+1}\mathbf{L}_{j+1,j}$
\STATE $\mathbf{x}_{j}=\mathbf{x}_{j-1}+\overline{\mathbf{d}}_j\overline{\bm{\gamma}}_j$, $\mathbf{r}_j=\mathbf{r}_{j-1}-\mathbf{A}^{'}\overline{\mathbf{d}}_j\overline{\bm{\gamma}}_j$
 \STATE If $\|\mathbf{r}_j\|/{\bm{\beta}}\leq \varepsilon$, then Stop		
	\end{algorithmic}
\end{algorithm}

\section{Numerical Experiments}\label{section-5}
In this section, we present some numerical experiments on model problems, including Chen's chaotic attractor, and color image deblurred problems, and demonstrate the robustness and effectiveness of the proposed algorithms compared with QBiCG when applied to solve the quaternion linear systems \eqref{1-1}.  All experiments have been carried out in Matlab 2022a on a personal computer with Inter(R) Core(TM) i7-12700KF @3.6GHz and 16.00 GB memory. Denote by ``IT'', the number of iterations, by ``CPU'', the elapsed computing time in seconds. Unless otherwise stated, the stopping criteria of all tested methods are either the $m$-th iteration relative residual satisfying
$$
\mbox{RR}:=\|\mathbf{b}-\mathbf{A}\mathbf{x}_m\|_2/\|\mathbf{r}_0\|_2\leq 10^{-7},
$$
or the maximum number of iterations reaching 5000. In all tests, we take the initial guess $\mathbf{x}_0$ as the zero vector, and $\mathbf{M}=\mathbf{M}_1=(\mathbf{D}+\mathbf{L})\mathbf{D}^{-1}(\mathbf{D}+\mathbf{U})$ as the left preconditioner, i.e., the SSOR preconditioner, which is based on a decomposition of the coefficient matrix $\mathbf{A}$ into a nonsingular diagonal matrix $\mathbf{D}$, a strictly lower triangular matrix $\mathbf{L}$, and a strictly upper
triangular matrix $\mathbf{U}$, such that $\mathbf{A}=\mathbf{D}+\mathbf{L}+\mathbf{U}$.

 \begin{example}\label{ex-1}
{\rm In this example, we consider solving the original quaternion linear systems \eqref{1-1} with its parameters given by
\begin{equation}\label{5-1}
\begin{aligned}
\mathbf{A}&=A_0+A_1\mathbf{i}+A_2\mathbf{j}+A_3\mathbf{k},\\
\mathbf{b}&=b_0+b_1\mathbf{i}+b_2\mathbf{j}+b_3\mathbf{k},
\end{aligned}
\end{equation}
where $b_i=\mbox{rand}(n,1)$, $i=1,2,3,4$, is an $n$-dimensional random vector with entries uniformly distributed in $[0,1]$, $A_1=2*A_0$, $A_2=-1.5*A_0$, $A_3=0.5*A_0$, and $A_0$ from the University of Florida Sparse Matrix Collection \cite{Davis1} are chosen as $A_0=\textbf{mcca}$, $\textbf{cavity01}\ $, $\textbf{rdb968}$ and $\textbf{pde2961}$, respectively.}
\end{example}

{\rm We first implemented five iterative solvers, including QBiCG, Algorithms \ref{algorithm-5}, \ref{algorithm-6} and their preconditioned variants, for solving the cases $\textbf{mcca}$, $\textbf{cavity01}$, $\textbf{rdb968}$ and $\textbf{pde2961}$, and then recorded the numerical results in Table \ref{tab-1}. It turns out that the iterations and computing time in seconds required by Algorithms \ref{algorithm-6}, \ref{algorithm-7}, and \ref{algorithm-8} are commonly less than that of QBiCG and Algorithm \ref{algorithm-5}, where Algorithms \ref{algorithm-7} and \ref{algorithm-8} outperform the other methods, with they having similar performance. Moreover, we observe that the convergence rate of the two preconditioned variants is at least five times that of the other solvers.
Convergence histories of all tested solvers for the cases $\bf{mcca}$ and $\bf{pde2961}$ were depicted in Fig. \ref{fig:1}. As shown in this figure, the convergence curve of QBiCG seems considerably irregular, whereas all proposed algorithms often present much smoother convergence behaviors than the QBiCG solver. Furthermore, both Algorithms \ref{algorithm-7} and \ref{algorithm-8} yield rather monotone convergence behaviors compared
to the other methods. Additionally, although the QBiCG solver performs better than Algorithm \ref{algorithm-5} for most cases regarding the iterations and computing time, the latter still converges much smoother than the former solver.}
\begin{table}[htbp]
	\footnotesize
	\caption{Numerical comparison results of Example \ref{ex-1}. \label{tab-1}}
	\begin{center}
		\begin{tabular}{|c|c|c|c|c|c|c|} \hline
			Data set& Application discipline &$n$ & Algorithms & IT& CPU& RR \\ \hline	
&&&QBiCG&{1194}&1.25& 1.91e-06\\
   &&&Algorithm \ref{algorithm-5}&1432&1.76&   1.83e-06\\
   \bf{mcca}   &\scriptsize{2D/3D Problem} &$180$&Algorithm \ref{algorithm-6} &889& 0.78&  1.47e-06 \\
         &&&Algorithm \ref{algorithm-7} &14&0.03&   5.61e-07\\
            &&&Algorithm \ref{algorithm-8} &14&0.03&  5.60e-07\\ \hline
			&&&QBiCG&{892}&2.13& 7.62e-07\\
   &&&Algorithm \ref{algorithm-5}&938&  2.52&  9.95e-07\\
   \bf{bfwa398}   & \scriptsize{Electromagnetics Problem}&$398$&Algorithm \ref{algorithm-6} &875& 2.05&9.75e-07 \\
            &&&Algorithm \ref{algorithm-7} &53&0.28& 9.44e-07\\ 
              &&&Algorithm \ref{algorithm-8} &51&0.24&   9.08e-07\\ \hline
&&&QBiCG&{3829}&48.11&9.05e-07\\
   &&&Algorithm \ref{algorithm-5}&{4731}&80.21& 9.97e-07 \\
     \bf{rdb968} &\scriptsize{Computational Fluid Dynamics}&$968$&Algorithm \ref{algorithm-6} &3765&44.47&9.87e-07 \\
         &&&Algorithm \ref{algorithm-7} &121&7.80& 5.25e-07\\
            &&&Algorithm \ref{algorithm-8} &117&7.32& 4.54e-07\\ \hline     
  &&&QBiCG&2687&142.55& 9.95e-07\\
   &&&Algorithm \ref{algorithm-5}&{3121}& 157.87&  9.98e-07 \\
      \bf{pde2961}&\scriptsize{2D/3D Problem}&$2961$&Algorithm \ref{algorithm-6} &2551&139.35&  9.66e-07\\
           &&&Algorithm \ref{algorithm-7} &190&16.64&  9.74e-07\\ 
            &&&Algorithm \ref{algorithm-8} &187& 16.38& 9.51e-07\\ \hline
		\end{tabular}
	\end{center}
\end{table}
\begin{figure}[htbp]
  \centering
  \begin{minipage}[c]{1.0\textwidth}
			\includegraphics[width=2.49in]{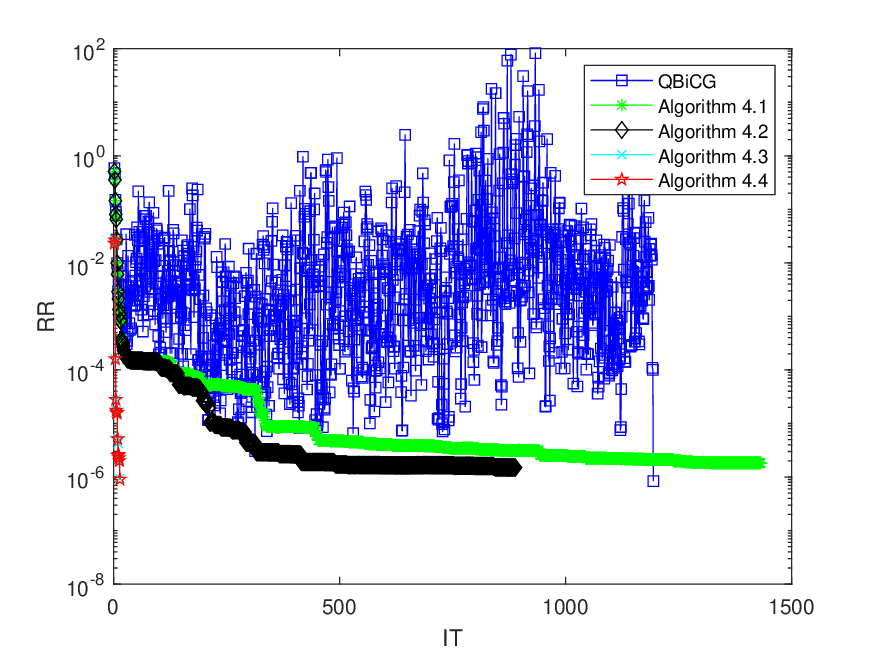}\quad 
   \includegraphics[width=2.49in]{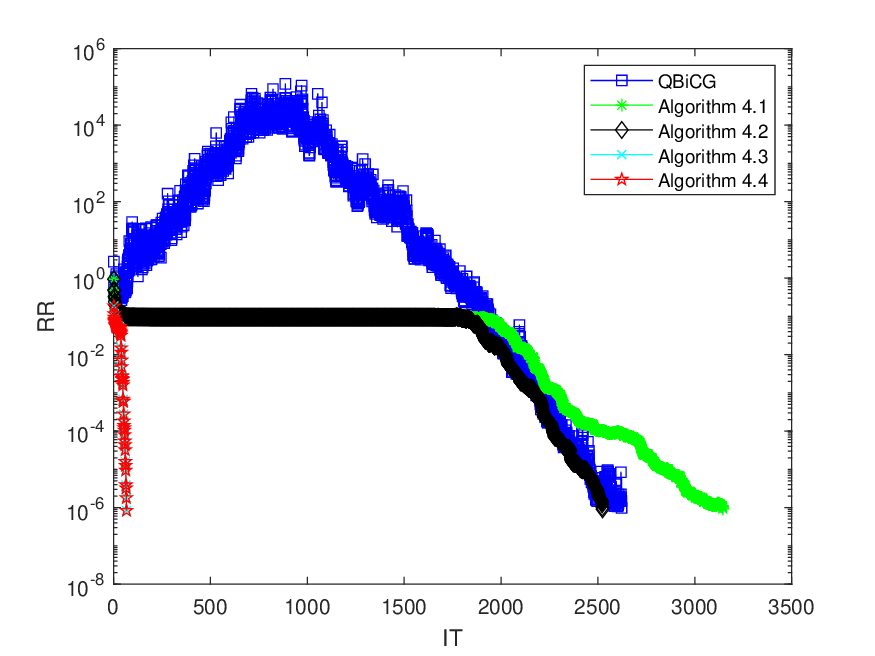}
		\end{minipage}
  \caption{
  Convergence histories for Example \ref{ex-1} with $\bf{mcca}$ (left) and $\bf{pde2961}$ (right).}
  \label{fig:1}
\end{figure}

It is well-known that various signals of engineering problems are always collected by sensors in three dimensions. Herein, we denote a three-dimensional signal by a pure quaternion function of time $\mathbf{x}(t)={x}_r(t)\mathbf{i}+{x}_g(t)\mathbf{j}+{x}_b(t)\mathbf{k}$, in which ${x}_r(t),{x}_g(t)$, and ${x}_b(t)$ represent real functions, such as red, green and blue channels, respectively. Observe from \cite{Jia4} that, the collected signals $\mathbf{x}(t-p),\cdots,\mathbf{x}(t)$, regarded as input signals, often act on the desired quaternion filters $\mathbf{w}(0),\cdots,\mathbf{w}(p)$ with $\mathbf{w}(s)=\mathbf{w}(s)_0+{w}_r(s)\mathbf{i}+{x}_g(s)\mathbf{j}+{x}_b(s)\mathbf{k}$, $s=0,\cdots,p$, such that the filtered output can match the target signal $\mathbf{y}(t)$, i.e.,
\begin{equation}\label{5-2}
\mathbf{y}(t)=\sum_{s=0}^{p}\mathbf{x}(t-s)\mathbf{w}(s).
\end{equation}
If we now take $q$ signal data along the time series, this model can be reformulated as the following quaternion linear systems
\begin{equation}\label{5-3}
\mathbf{X}*\mathbf{w}=\mathbf{y},
\end{equation}
where $$\begin{aligned}
\mathbf{X}&=\begin{bmatrix}
\mathbf{x}(t)&\mathbf{x}(t-1)&\mathbf{x}(t-2)&\cdots&\mathbf{x}(t-p)\\
\mathbf{x}(t+1)&\mathbf{x}(t)&\mathbf{x}(t-1)&\cdots&\mathbf{x}(t-p+1)\\
\vdots&\vdots&\ddots&\vdots&\vdots\\
\mathbf{x}(t+q)&\mathbf{x}(t+q-1)&\mathbf{x}(t+q-2)&\cdots&\mathbf{x}(t+q-p)\\
\end{bmatrix},\\
\mathbf{w}&=[\mathbf{w}(0) \quad \mathbf{w}(1) \quad \mathbf{w}(2)\quad \cdots \quad\mathbf{w}(p)]^T,\\
\mathbf{y}&=[\mathbf{y}(t)\quad \mathbf{y}(t+1) \quad \mathbf{y}(t+2) \quad\cdots \quad\mathbf{y}(t+q)]^T.\\
\end{aligned}$$
Let $\mathbf{X}=X_0+X_1\mathbf{i}+X_2\mathbf{j}+X_3\mathbf{k}\in \mathbb{Q}^{n\times n},$ $\mathbf{y}=y_0+y_1\mathbf{i}+y_2\mathbf{j}+y_3\mathbf{k}\in \mathbb{Q}^{n}$ and $\mathbf{w}=w_0+w_1\mathbf{i}+w_2\mathbf{j}+w_3\mathbf{k}\in \mathbb{Q}^{n}$. Then, we can utilize the proposed algorithms and QBiCG for solving the quaternion linear systems \eqref{5-3}.

\begin{example}\label{ex-2}
{\rm Consider Chen's chaotic attractor \cite{Chen00, Chen11}, i.e., three-dimensional nonlinear system, applied in dynamical systems and chaos control and synchronization. The attractor is quite similar to those of the Lorenz system used in atmospheric convection \cite{Sparrow}, but they are topologically not equivalent. The Chen chaotic system, via a simple state feedback to the second equation in the Lorenz system, is formulated as the following coupled differential equations
\begin{equation}\label{5-4}
\frac{\partial x}{\partial t}=\alpha(y-x),\quad \frac{\partial y}{\partial t}=(\rho-\alpha)x-xz+\beta y,\quad \frac{\partial z}{\partial t}=xy-\beta z,
\end{equation}
where $\alpha,\beta$ and $\rho$ are positive. Here the parameters are chosen as $\alpha=35,\beta=3$ and $\rho=28$, and then solve the above equations \eqref{4-2} by the MATLAB order $\text{ODE45}(f(t, [x, y, z]),$ $ [0, T],[1,1,1])$ with different $T>0$. Let the solutions $y_r(t),y_g(t)$ and $y_b(t)$ of Eqs.\eqref{5-2} be the target signal 
$$
\mathbf{y}(t)=y_r(t)\mathbf{i}+y_g(t)\mathbf{j}+y_b(t)\mathbf{k},
$$
and let 
$$\mathbf{x}(t)=y_r(t-1)\mathbf{i}+y_g(t-1)\mathbf{j}+y_b(t-1)\mathbf{k}+\mathbf{n}(t)$$
be the input signal with $\mathbf{n}(t)$ being a random noise. This leads to a quaternion linear system built as Eq.\eqref{5-3}. }

{\rm In Table \ref{tab-2}, for Algorithms \ref{algorithm-5}-\ref{algorithm-8} and QBiCG applied to solve Eq.\eqref{5-3}, we listed the obtained numerical results for different dimensions. It appears from the results that, as the size increases, the iterations and the elapsed CPU time corresponding to Algorithms \ref{algorithm-5}, \ref{algorithm-6}, and the QBiCG solver rapidly increased, while Algorithms \ref{algorithm-7} and \ref{algorithm-8} gradually
increased. These comparison results demonstrate the superiority of Algorithms \ref{algorithm-7} and \ref{algorithm-8}, and Algorithm \ref{algorithm-6} based on coupled two-term recurrences is still slightly better than QBiCG in terms of iterations and the computing time.
Besides, we displayed the convergence histories of all algorithms for the case $n=101$ in Fig. \ref{fig:2}. One can see that the four proposed algorithms work more smoothly than QBiCG, whereas the convergence behavior of the latter is still less robust. On the other hand, we can also see that, both Algorithms \ref{algorithm-7} and \ref{algorithm-8} have almost the same convergence behaviors, but the the former is more costly.
}
\end{example}

\begin{table}[htbp]
	\footnotesize
	\caption{Numerical comparison results of Example \ref{ex-2} \label{tab-2}}
	\begin{center}
		\begin{tabular}{|c|c|c|c|c|} \hline
			\bf Dimension &\bf Algorithms & \bf CPU&\bf IT &\bf RR \\ \hline
   & QBiCG &0.10&150& 5.47e-07\\
   &Algorithm \ref{algorithm-5}&0.14&170& 6.57e-07\\
    49&Algorithm \ref{algorithm-6}&0.09&145&4.04e-07\\
	&Algorithm \ref{algorithm-7}&0.03&90&9.68e-07\\
    &Algorithm \ref{algorithm-8}&0.02&87&  7.52e-07\\\hline
			& QBiCG &0.85&827& 5.66e-07\\
   &Algorithm \ref{algorithm-5}& 1.46&1173& 9.98e-07\\
   101 &Algorithm \ref{algorithm-6}& 0.74&753& 8.72e-07\\
	&Algorithm \ref{algorithm-7}&0.22&378&9.68e-07\\
    &Algorithm \ref{algorithm-8}&0.17&344&8.84e-07\\		\hline		
    & QBiCG &4.55&1705&  9.93e-07 \\
   &Algorithm \ref{algorithm-5}&7.16&2146&  9.98e-07\\
    203&Algorithm \ref{algorithm-6}&3.41&1453& 9.94e-07\\
	&Algorithm \ref{algorithm-7}&1.36&989&9.96e-07\\
    &Algorithm \ref{algorithm-8}& 0.73&746&9.95e-07\\		\hline
    & QBiCG &23.51&3266& 8.82e-07 \\
   &Algorithm \ref{algorithm-5}&45.29&4013&9.91e-07\\
   303 &Algorithm \ref{algorithm-6}&19.32&2824& 5.08e-07\\
	&Algorithm \ref{algorithm-7}& 2.43&1382&9.98e-07\\
    &Algorithm \ref{algorithm-8}& 1.74&1096& 9.96e-07\\		\hline
		\end{tabular}
	\end{center}
\end{table}

\begin{figure}[htbp]
  \centering
  \begin{minipage}[c]{0.6\textwidth}
			\includegraphics[width=3in]{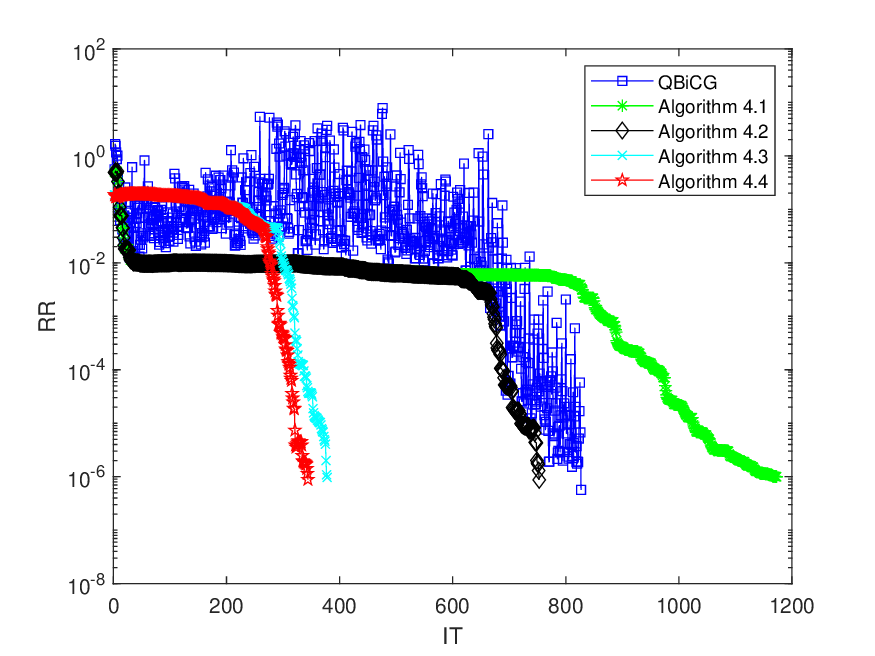}
		\end{minipage}
  \caption{Convergence histories of Example \ref{ex-2} with $n=101$.}
  \label{fig:2}
\end{figure}

In what follows, we will further test the effectiveness when the proposed algorithms are applied to solve the original linear systems \eqref{1-1} arising from the color image deblurred problems, and then compare them numerically with QBiCG.
\begin{example}\label{ex-3}
{\rm Let the pure quaternion matrix $\hat{\mathbf{X}}=\hat{{X}}_1\mathbf{i}+\hat{X}_2\mathbf{j}+\hat{X}_3\mathbf{k}\in \mathbb{Q}^{n\times n}$ be a color image, in which $\hat{X}_1, \hat{X}_2, \hat{X}_3\in \mathbb{R}^{n\times n}$ represent the red, green and blue channels, respectively. It consists of $n\times n$ color pixel values in the range [0, d] with $d = 255$ being the maximum possible pixel value of the image. Denote by $\hat{\mathbf{x}}=\mbox{vec}(\hat{\mathbf{X}})\in \mathbb{Q}^{n^2}$,  a pure quaternion vector generated by the stretching operator $\mbox{vec}(\cdot)$ that stacks the columns of the original color image $\hat{\mathbf{X}}$. If there exists a blurring matrix $\mathbf{A}$ on the color image, such that the filtered output is the observed color image $\mathbf{b}\in\mathbb{Q}^{n^2}$, then this model can be expressed mathematically as the quaternion linear systems \eqref{1-1}.

 The performance of all the tested algorithms is evaluated by four measures the computing time ($\mathbf{CPU}$), the peak signal-to-noise ratio ($\mathbf{PSNR}$) in decibel ($\mathbf{dB}$), structural similarity ($\mathbf{SSIM}$) \cite{Wang0101}, relative residual ($\mathbf{RR}$). Define
$$\begin{aligned}
\mathbf{PSNR}(\mathbf{X})&=10\log_{10}(\frac{3mnd^2}{\|\hat{\mathbf{x}}-\mathbf{x}\|_2^2}),\quad \mathbf{RR}(\mathbf{X})=\frac{\|\hat{\mathbf{x}}-\mathbf{x}\|_2}{\|\hat{\mathbf{x}}\|_2},  \\
\mathbf{SSIM}(\mathbf{X})&=\frac{(2\mu_{\hat{\mathbf{x}}}\mu_{\mathbf{x}_k}+c_1)(2\sigma_{\hat{\mathbf{x}}\mathbf{x}_k}+c_2)}{(\mu_{\hat{\mathbf{x}}}^2+\mu_{\mathbf{x}_k}^2+c_1)(\sigma_{\hat{\mathbf{x}}}^2+\sigma_{\mathbf{x}_k}^2+c_2)},
\end{aligned}
$$
where ${\mathbf{x}}=\mbox{vec}({\mathbf{X}})$ is an $n^2$-dimensional quaternion vector with ${\mathbf{X}}$ being the restored color image, $\mu_{(\hat{\mathbf{x}})}$, $\mu_{(\mathbf{x}_k)}$ and $\sigma_{(\hat{\mathbf{x}})}$, $\sigma_{(\mathbf{x}_k)}$  are the means and variances of  $\hat{\mathbf{x}}$ and $\mathbf{x}_k$, respectively, $\sigma_{\hat{\mathbf{x}}\mathbf{x}_k}$ is the covariance between $\hat{\mathbf{x}}$ and $\mathbf{x}_k$, $c_1=(0.01L)^2, c_2=(0.03L)^2$ with $L$ being the dynamic range of pixel values. Herein, we set the stopping criterion as the maximum iterations achieving 200. Since the arithmetic operations of the inverse of the left preconditioner are expensive, we only compare Algorithms \ref{algorithm-5} and \ref{algorithm-6} with QBiCG \cite{Li1} in this example. 

\begin{table}[htbp]
	\footnotesize
	\caption{Numerical results of Example \ref{ex-3} with single channel deblurring. \label{tab-3}}
 \renewcommand\arraystretch{1}
	\begin{center}
		\begin{tabular}{|c|c|c|c|c|c|} \hline
			\bf Case &\bf Algorithm &\bf PSNR & \bf SSIM &\bf CPU&\bf RR \\  \hline
   &QBiCG&29.87& 0.93&226.52&2.11e-02\\
   \textbf{Mudan}& Algorithm\ref{algorithm-5}& 28.62& 0.91& 273.21&2.92e-02\\
& Algorithm\ref{algorithm-6}&30.11&0.94&204.63&1.61e-02\\ 
 \hline
 &QBiCG&30.96& 0.95&134.68&1.85e-02\\
   \textbf{Sunset}& Algorithm\ref{algorithm-5}& 29.59& 0.93& 159.27&2.76e-02\\
& Algorithm\ref{algorithm-6}&31.35&0.96&121.92&1.46e-02\\ 
 \hline
 &QBiCG&34.53& 0.96&260.43&1.63e-02\\
   \textbf{Desk}& Algorithm\ref{algorithm-5}&  33.87& 0.95&315.84 &2.02e-02\\
& Algorithm\ref{algorithm-6}&35.14&0.97&231.47&1.42e-02\\ 
 \hline
		\end{tabular}
	\end{center}
\end{table}

Consider two kinds of deblurred color image problems. One is to restore the color images blurred by a single channel, e.g., the blurring matrix is set to be $\mathbf{A}={A}_0\in \mathbb{R}^{n^2\times n^2}$. As shown in \cite{Bouhamidi12,Li1212}, we set ${A}_0=B^{(1)}\otimes B^{(2)}$ as the blurring matrix, in which $B^{(1)}=(b_{ij}^{(1)})_{1\leq i,j\leq n}$ 
 and $B^{(2)}=(b_{ij}^{(2)})_{1\leq i,j\leq n}$ are the Toeplitz matrices with their elements given by
 \begin{equation}\label{5-5}
b_{ij}^{(1)}=\left\{\begin{aligned}	
&\frac{1}{\sigma\sqrt{2\pi}}\exp(-\frac{(i-j)^2}{2\sigma^2}),\ \ i-j\leq r,\\
&\frac{1}{\pi r^2},\ \ j-i< r,\\
&0,\ \ \ \  \mbox{otherwise},
\end{aligned}
\right.\ \ \ \ \ b_{ij}^{(2)}=\left\{\begin{aligned}	
	&\frac{1}{2s-1},\ \ |i-j|\leq s,\\
	&0,\ \ \ \  \mbox{otherwise},
\end{aligned}
\right.
\end{equation}
where $\sigma=1$, $r=10$, and $s=7$. 

We applied Algorithms \ref{algorithm-5}, \ref{algorithm-6} and QBiCG to deal with three color images blurred by a single channel blurring matrix $\mathbf{A}={A}_0$ in \eqref{5-5}, including ``\textbf{Mudan}'', ``\textbf{Sunset}'' and ``\textbf{Desk}'' of size $100\times 100$ in JPG format. All of them only have three channels, i.e., red, green, and blue channels. In Table \ref{tab-3}, we listed the computing time, PSNR, and SSIM of the recovered
images and the relative error when the three algorithms reached 300 iterations. We see from this table that the PSNR, SSIM, and computing times required by QBiCG are slightly lower than Algorithm \ref{algorithm-5}, but more than Algorithm \ref{algorithm-6}. This means that the deblurring quality corresponding to Algorithm \ref{algorithm-6} outperforms other solvers. The significant performance
advantage of Algorithm \ref{algorithm-6} is due to its numerical stability, such that it takes less time to achieve the required accuracy. Moreover, we depicted the three deblurred color images for all compared methods in Fig.\ref{fig:3}.

To further show the effectiveness of our algorithms, we can recover three color images, ``\textbf{House}'', ``\textbf{Rose}'' and ``\textbf{Sloth}'' of size $100\times 100$ in PNG format. It is worth noticing that these color images containing four channels, i.e., transparency, red, green and blue channels, can be characterized by a quaternion matrix $\hat{\mathbf{X}}=\hat{{X}}_0+\hat{{X}}_1\mathbf{i}+\hat{X}_2\mathbf{j}+\hat{X}_3\mathbf{k}\in \mathbb{Q}^{n\times n}$. Assume now that the three color images are blurred by a multichannel blurring matrix $\mathbf{A}$, e.g., $\mathbf{A}=A_0+A_1\mathbf{i}+A_2\mathbf{j}+A_3\mathbf{k}\in \mathbb{Q}^{n^2\times n^2}$, where ${A}_0, A_1, A_2, A_3$ are given by
$$
{A}_0=B^{(2)}\otimes B^{(2)},\quad {A}_1={A}_0,\quad {A}_2=-{A}_0,\quad {A}_3=-{A}_0,
$$
and $s=3$.

\begin{table}[htbp]
	\footnotesize
	\caption{Numerical results of Example \ref{ex-3} with multichannel deblurring. \label{tab-4}}
 \renewcommand\arraystretch{1}
	\begin{center}
		\begin{tabular}{|c|c|c|c|c|c|} \hline
			\bf Case &\bf Algorithm &\bf PSNR & \bf SSIM &\bf CPU&\bf RR \\  \hline  
 &QBiCG&26.89& 0.82&732.09 &2.45e-02\\
   \textbf{House}& Algorithm\ref{algorithm-5}& 26.23& 0.81& 913.19&2.79e-02\\
& Algorithm\ref{algorithm-6}&27.15&0.83&665.78&2.20e-02\\ 
 \hline
  &QBiCG&27.39& 0.85&926.54&2.60e-02\\
   \textbf{Rose}& Algorithm\ref{algorithm-5}& 27.02& 0.84&1178.67&2.81e-02\\
& Algorithm\ref{algorithm-6}&27.95&0.86&803.45&2.43e-02\\ 
 \hline
  &QBiCG&  28.01& 0.87&897.19&2.51e-02\\
   \textbf{Sloth}& Algorithm\ref{algorithm-5}&27.82& 0.86&1034.16&2.82e-02\\
& Algorithm\ref{algorithm-6}&28.54&0.98&811.25&2.31e-02\\ 
 \hline
		\end{tabular}
	\end{center}
\end{table}

The results of the numerical comparison for QBiCG, Algorithms \ref{algorithm-5} and \ref{algorithm-6} were reported in Table \ref{tab-4} and Fig.\ref{fig:4}. As seen in Fig.\ref{fig:4}, all tested algorithms for restoring the deblurred color images are feasible, in which Algorithm \ref{algorithm-6} seems to recover with higher quality than other solvers. Moreover, from Table \ref{tab-4}, we obtain that the PSNR, SSIM, and the consuming time to Algorithm \ref{algorithm-6} are still less than Algorithm \ref{algorithm-5}, even though they are 
mathematically equivalent, but the former is more robustness. The main reason is that Algorithm \ref{algorithm-5}, based on the three-term recurrences, is often sensitive to the rounding error.}
\end{example}

\begin{figure}[!htbp]
	\hspace{-15em}
	\begin{center}
		\begin{minipage}[c]{1.0\textwidth}
			\includegraphics[width=5in]{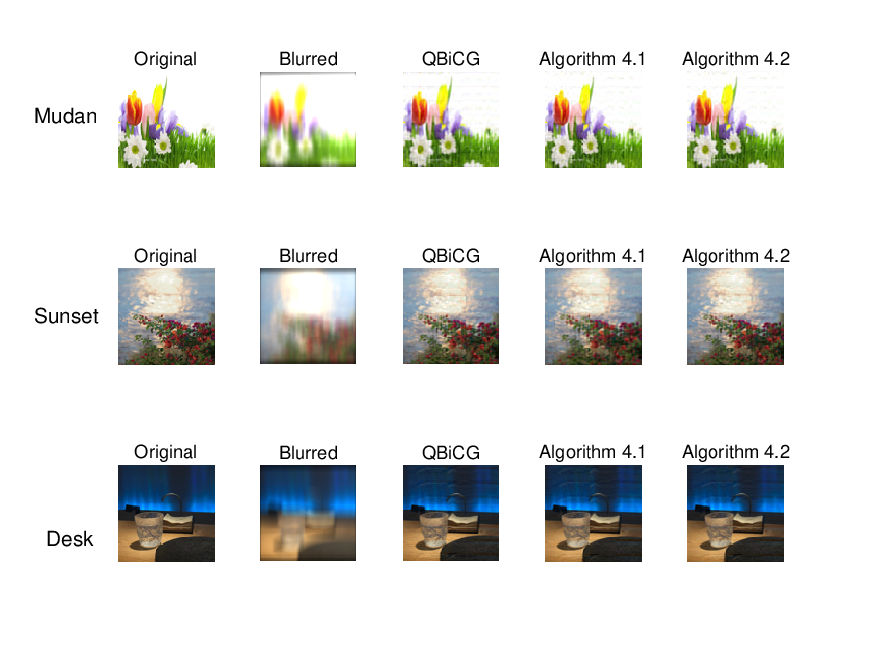}
		\end{minipage}
	\end{center}
	\vspace{-3em}\caption{Left to Right: the original $100\times 100$ images, the blurred images with
 single channel, the deblurred images were restored by QBiCG, Algorithms \ref{algorithm-5} and \ref{algorithm-6}, respectively.\label{fig:3}}
\end{figure}

\begin{figure}[!htbp]
	\hspace{-15em}
	\begin{center}
		\begin{minipage}[c]{1.0\textwidth}
			\includegraphics[width=5in]{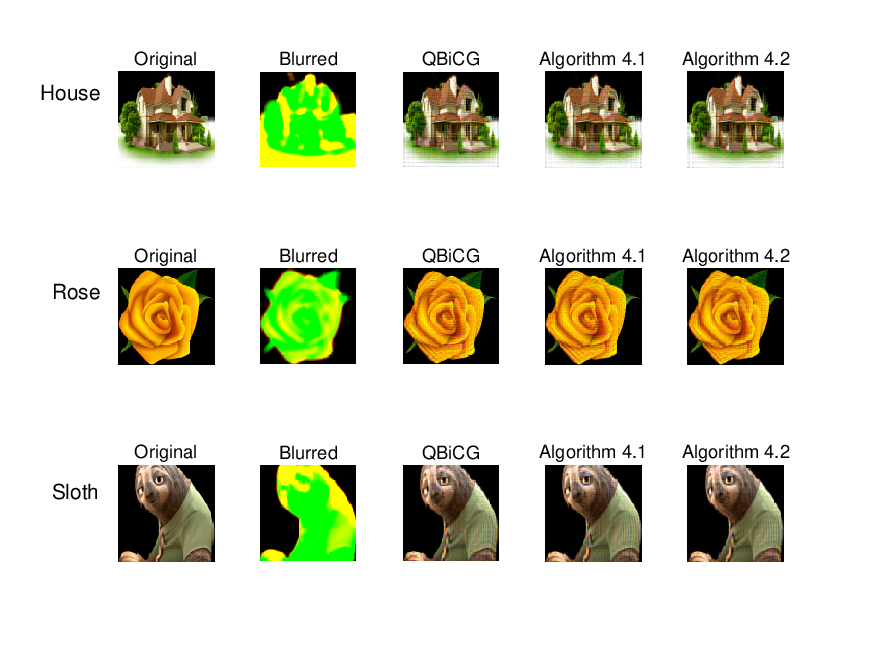}
		\end{minipage}
	\end{center}
	\vspace{-3em}\caption{Left to Right: the original $100\times 100$ images, the blurred images with
 multichannel, the deblurred images were restored by QBiCG, Algorithms \ref{algorithm-5} and \ref{algorithm-6}, respectively.\label{fig:4}}
\end{figure}

\section{Conclusions}
\label{section-6}
In this paper, we derived four structure-preserving quaternion Krylov algorithms, including the QQMR method based on three-term recursions, the QQMR method based on two-term recursions and their preconditioned variants, for solving the quaternion linear systems \eqref{1-1}. The main contributions are the following:
\begin{itemize}
\item Two quaternion biconjugate orthonormalization (QBIO) processes are newly proposed for obtaining a pair of biorthonormal bases to quaternion Krylov subspaces $\mathcal{K}_m(\mathbf{A},\mathbf{v}_1)$ and $\mathcal{K}_m(\mathbf{A}^*,\mathbf{w}_1)$. Both inherit the algebraic symmetry of the real counterparts, saving about three-quarters of theoretical costs compared with the same processes over $\mathbb{R}$.
\item  QQMR method based on three-term recurrences and its robustness variant, as well as their preconditioned variants, are first developed for solving the quaternion linear systems. The convergence analysis of which is also newly established.
\item The proposed algorithms for solving the three-dimensional signal filtering and color image deblurring problems are more effective compared with QBiCG.
\end{itemize}



\end{document}